\documentclass[a4paper,oneside,11pt]{article}

\usepackage{amsmath,amsfonts,amscd,amssymb}
\usepackage{longtable,geometry}
\usepackage[english]{babel}
\usepackage[utf8]{inputenc}
\usepackage[active]{srcltx}
\usepackage[T1]{fontenc}
\usepackage{graphicx}
\usepackage{pstricks}
\usepackage{bbm}
\usepackage{mathtools}
\usepackage{hyperref}
\usepackage{enumitem}
\usepackage{MnSymbol}
\usepackage{stmaryrd}
\usepackage{nicefrac}

\usepackage{xcolor}
\usepackage{framed}

\colorlet{shadecolor}{blue!15}

\usepackage{amsthm}
\newtheorem{theorem}{Theorem}

\newtheorem{corollary}[theorem]{Corollary}
\newtheorem{lemma}[theorem]{Lemma}
\newtheorem{proposition}[theorem]{Proposition}

\newtheorem{remark}[theorem]{Remark}

\newcommand{\calA}{\mathcal{A}}

\newcommand{\calC}{\mathcal{C}}

\newcommand{\calE}{\mathcal{E}}
\newcommand{\calF}{\mathcal{F}}
\newcommand{\calG}{\mathcal{G}}
\newcommand{\calH}{\mathcal{H}}

\newcommand{\calL}{\mathcal{L}}

\newcommand{\calR}{\mathcal{R}}

\newcommand{\calT}{\mathcal{T}}

\newcommand{\calV}{\mathcal{V}}

\newcommand{\bbE}{\mathbb{E}}

\newcommand{\bbG}{\mathbb{G}}
\newcommand{\bbH}{\mathbb{H}}

\newcommand{\bbN}{\mathbb{N}}

\newcommand{\bbR}{\mathbb{R}}

\newcommand{\bbV}{\mathbb{V}}

\newcommand{\bbZ}{\mathbb{Z}}

\newcommand{\lr}[1][]{\xleftrightarrow{\: #1 \:\:}}

\newcommand{\mon}{\eqref{eq:4}}

\numberwithin{equation}{section}

\newcommand{\rk}[1]{\bgroup\color{red}%
  \par\medskip\hrule\smallskip%
  \noindent\textbf{#1}%
  \par\smallskip\hrule\medskip\egroup}

\title{Renormalization of crossing probabilities in the planar random-cluster model}
\author{Hugo Duminil-Copin\thanks{Universit\'e de Gen\`eve} \thanks{Institut des Hautes \'Etudes Scientifiques}, Vincent Tassion\thanks{ETH Zurich}}
\date{\today}

\begin{document}
\maketitle
 
\begin{abstract}
The study of crossing probabilities -- i.e.~probabilities of existence of paths crossing rectangles -- has been at the heart of the theory of two-dimensional percolation since its beginning. They may be used to prove a number of results on the model, including speed of mixing, tails of decay of the connectivity probabilities, scaling relations, etc. In this article, we develop a renormalization scheme for crossing probabilities in the two-dimensional random-cluster model. The outcome of the process is a precise description of an alternative between four behaviors: 
\begin{itemize}
\item {\em Subcritical:} Crossing probabilities, even with favorable boundary conditions, converge exponentially fast to 0. 
\item {\em Supercritical:} Crossing probabilities, even with unfavorable boundary conditions, converge exponentially fast to 1. 
\item {\em Critical discontinuous:} Crossing probabilities converge to 0 exponentially fast with unfavorable boundary conditions and to 1 with favorable boundary conditions.
\item {\em Critical continuous:} Crossing probabilities remain bounded away from 0 and 1 uniformly in the boundary conditions. 
\end{itemize}
The approach does not rely on self-duality, enabling it to apply in a much larger generality, including the random-cluster model on arbitrary graphs with sufficient symmetry, but also other models like certain random height models. 
\end{abstract}

\section{Introduction}

\subsection{Framework and Motivation}

In this article, we consider an infinite {\em biperiodic} (i.e.~invariant under the action of a $\mathbb{Z}^2$-isomorphic lattice) planar connected graph $\bbG$ with vertex-set $\bbV$ and edge-set $\bbE$. We embed the graph in such a way that $0$ is a vertex of $\bbG$ and translations by vectors $x\in\mathbb Z^2$ leave the graph invariant. The graph $\bbG$ is also assumed to be {\em invariant under $\pi/2$-rotations and reflections with respect to the $x$ and $y$ axis}.
 Below, $G$ will always refer to a finite subgraph of $\bbG$ with vertex-set $V$ and edge-set $E$. The {\em boundary} of $G$, denoted by $\partial G$, is the set of vertices in $G$ having at least one neighbor (in the sense of $\mathbb G$) outside $G$.
  
Percolation was introduced in the middle of the twentieth century to describe mathematically the inside of a porous material. While the model was originally motivated by an applied problem, it soon became a major object of interest in probability and mathematical physics. In a bond percolation model, each edge $e\in  E$ is either {\em open} or {\em closed}, a fact which is encoded by a function $\omega=(\omega_e:e\in E)$ from $E$ to $\{0,1\}$, where $\omega_e$ is equal to 1 if the edge $e$ is open, and $0$ if it is closed. A bond percolation model then consists in choosing edges of $G$ to be open or closed at random. 


The simplest and oldest model of bond percolation, called {\em Bernoulli percolation}, was introduced by Broadbent and Hammersley \cite{BroHam57}. In this model, each edge of $\mathbb G$ is open with probability $p$ in $[0,1]$ and therefore closed with probability $1-p$, independently of the state of other edges. Equivalently, the $\omega_e$ for $e\in \mathbb E$ are independent Bernoulli random variables of parameter $p$. This model has been intensively studied over the last sixty years, see e.g.~\cite{Gri99a}. While the theory of Bernoulli percolation still contains major open problems, it is fair to say that mathematicians are now in possession of a deep understanding of the model, especially in two dimensions (see \cite{BefDum13} and references therein). 

In recent years, more general percolation models appeared in various areas of statistical physics as natural models associated with other random systems (for instance spin models such as Ising and Potts models). While Bernoulli percolation is a product measure, the states of edges in these percolation models are typically not independent random variables. 

A large number of techniques developed for Bernoulli percolation do not extend to more general models, in particular due to the lack of independence. For this reason, the understanding of classical two-dimensional problems remained limited for more than thirty years, before improving in the last ten years.

The typical example of a dependent percolation model is provided by the random-cluster model, also called Fortuin-Kasteleyn percolation, which was introduced by Fortuin and Kasteleyn \cite{ForKas72} as a class of percolation models satisfying specific series and parallel laws. 
It is related to many other models of statistical mechanics, including the Potts model. 
We direct the reader to the monograph \cite{Gri06}  for background on the random-cluster model, and to the lecture notes \cite{Dum17a} for an exposition of the most recent results.
 
In order to define the model, we first consider a finite subgraph $G$ of $\mathbb G$.  The {\em boundary conditions} $\xi$ on $G$ are given by a partition of $\partial G$. Two vertices of $G$ are {\em wired together} if they belong to the same element of the partition $\xi$. The {\em free} (resp.\@ {\em wired}) boundary conditions, denoted by $\xi=0$ (resp. $\xi=1$) refer to boundary conditions in which no two (resp.~all) boundary vertices are wired together. 

The random-cluster measure on $G$ with edge-weight $p\in[0,1]$, cluster-weight $q>0$ and boundary conditions $\xi$ is given by
\begin{equation}\label{eq:RCM_def1}
	\phi_{G}^\xi[\omega]=\frac{1}{Z^\xi_G}\: \Big(\frac p{1-p}\Big)^{|\omega|}\:q^{k_\xi(\omega)},
\end{equation}
where $|\omega|$ is the number of open edges in $\omega$, $k_\xi(\omega)$ is the number of connected components of the graph obtained from $\omega$ by identifying wired vertices together, and finally $Z^\xi_G$ is a normalizing constant called the partition function, chosen in such a way that $\phi_{G}^\xi$ is a probability measure. 

There is a natural notion of infinite-volume random-cluster measure. More precisely, let $\mathcal B_{G,\xi}$ be the event that two vertices of $\partial G$ are connected to each others in the configuration outside $G$ if and only if they are in the same element of the partition $\xi$. A DLR-random-cluster measure $\phi$ on $\{0,1\}^\bbE$ is a measure satisfying
$$\phi[\:\cdot\,_{|E}\:|\:\mathcal B_{G,\xi}]=\phi_G^\xi$$
for every finite subgraph $G$ of $\bbG$ and every $\xi\in\{0,1\}^\bbE$ such that $\phi[\mathcal B_{G,\xi}]>0$.
One can construct DLR-random-cluster measures on $\mathbb G$ by taking the weak limit of measures on a sequence of finite subgraphs tending to $\mathbb G$. The two measures $\phi^0_{\bbG}$ and $\phi^1_{\bbG}$ denote the measures on $\mathbb G$ obtained by taking the weak limits of the random-cluster measures with free and wired boundary conditions respectively.

The theory of percolation in two dimensions relies heavily on the study of so-called crossing probabilities, i.e.~probabilities that rectangles contain a path of open edges crossing them (say from top to bottom, or from left to right). This study relies on two important (almost independent) pillars:
\begin{itemize}  
\item The first one, called the {\em RSW theory}, states that lower and upper bounds on crossing probabilities of rectangles of a certain aspect ratio imply similar bounds for crossing probabilities for rectangles of other aspect ratios.  The first result in this direction goes back to the seminal works of Russo \cite{Rus78} and Seymour and Welsh \cite{SeyWel78} (recently, several alternative proofs of the theorem have been obtained for Bernoulli percolation \cite{BolRio06,BolRio06c,BolRio10,Tas14b}).   The interest of this theorem is that it enables us to transfer lower bounds for crossing probabilities of very wide rectangles (which are usually easy to obtain) to lower bounds for crossing probabilities of very thin rectangles. By duality, it also enables one to transform upper bounds in thin rectangles into upper bounds in wide rectangles. This tool simplifies greatly the study of crossing probabilities, since one can choose the aspect ratio of the rectangles under consideration freely, and therefore adapt this choice to the problem at hand.

\item The second pillar is a \emph{renormalization of crossing probabilities}: by bounding the crossing probabilities at one scale in terms of the crossing probabilities at a lower scale, we obtain quantitative bounds on the crossing probabilities. Here, we make a slight abuse of terminology: in this context, we do not exhibit an exact renormalization flow (this would imply the existence of a scaling limit at criticality): we only work with inequalities that are sufficient to exhibit very explicit bounds on the crossing probabilities. While similar approaches were implemented for Bernoulli percolation in the past, the case of dependent percolation processes is substantially more subtle. Indeed, for Bernoulli percolation, renormalization inequalities are obtained using the fact that crossing probabilities in disjoint rectangles are independent. For dependent processes, crossing probabilities can be very different under different ``boundary conditions''. 
\end{itemize}

Some progress for the random-cluster model has been made in the two directions above. The RSW theorem was generalized to specific examples of dependent percolation models \cite{BefDum12}, and in weaker forms for very general models \cite{Tas14b}. The study of crossing probabilities that are uniform in boundary conditions was initiated in \cite{DumHonNol11} for the FK Ising model and then extended to critical random-cluster models with cluster-weight $q\ge1$ in \cite{DumSidTas14}. 

Despite this progress, the arguments developed so far are not usually working in a general framework and often rely on specific properties of the models that should a priori not be relevant to the problem at hand (the archetypical example would be exact integrability or exact self-duality of the model). In this article, we therefore propose robust arguments improving the understanding of the two aspects above:
\begin{itemize}
\item We prove a new RSW-result for the random-cluster model, without using exact self-duality at criticality.
\item We  perform two different types of renormalization procedures. The first one is very similar to the one of Bernoulli percolation (in fact it should maybe be called coarse graining rather than renormalization), and provides bounds on the crossing probabilities with favorable boundary conditions. The second one is new and allows us to study the effect of boundary conditions (in particular it provides bounds on the crossing probabilities with unfavorable boundary conditions).
\end{itemize}

Of course, we do not claim to tackle all percolation models of interest (see Section~\ref{sec:other models}), but we believe that this paper is a first step towards a comprehensive understanding of these issues.

\begin{figure}\begin{center}
\includegraphics[width=0.50\textwidth]{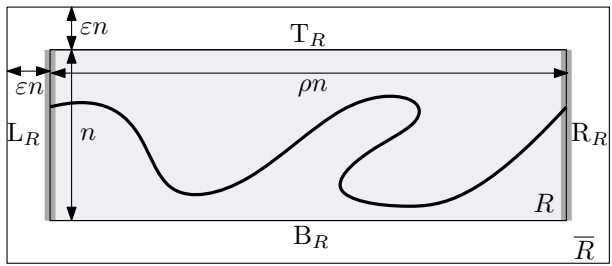}
\caption{\label{fig:1}The rectangles $R$ and $\overline R$ (this one is defined and used in Corollary~\ref{cor:z}), the different sides of $R$, as well as a crossing from left to right.}\end{center}
\end{figure}

\subsection{Russo-Seymour-Welsh theory}

As mentioned in the previous section, the study of crossing estimates relies on two pillars. We start by discussing the first one, namely the RSW theorem. 

For $n\ge1$, we define the box $\Lambda_n:= [-n,n]^2$ and the strip $S_n:=\mathbb R\times [-n,2n]$. We identify a subset $A$ of $\mathbb R^2$ with the subgraph of $\mathbb G$ induced by edges having at least one endpoint in $A$.
 For example, a rectangle $R:=[a,b]\times[c,d]$ is the subgraph of $\bbG$ induced by the edges with at least one endpoint in $R$; see Fig.~\ref{fig:1}. The quantities $b-a$ and $d-c$ are respectively called the {\em width} and the {\em height} of the rectangle.
 We denote by ${\rm T}_R$, ${\rm L}_R$, ${\rm B}_R$ and ${\rm R}_R$ the top, left, bottom and right sides of $R$. 
 
 The rectangle $R$ is {\em crossed horizontally}, denoted $\calH_R$, if $\omega\cap R$ contains a path of open edges (called {\em horizontal crossing}) from ${\rm L}_R$ to ${\rm R}_R$. Similarly, one defines the event that $R$ is {\em crossed vertically}, denoted $\calV_R$, if $\omega\cap R$ contains a path of open edges (called {\em vertical crossing}) from ${\rm B}_R$ to ${\rm T}_R$. 

 For $n\ge 1$, we consider three measures $\phi_{S_n}^0$, $\phi_{S_n}^1$ and $\phi_{S_n}^{0/1}$ in the infinite strip $S_n$ corresponding respectively to the measures with free (no boundary vertices are wired), wired (all boundary vertices are wired) and Dobrushin boundary conditions (all the vertices on the bottom of the strip are wired). These strip measures are formally defined in Section~\ref{sec:0}.

 The first result is the following.

\begin{theorem}[RSW]\label{thm:RSW}
Fix $p\in[0,1]$ and $q\ge1$. For every $\rho\ge1$, there exists an increasing homeomorphism $f=f_\rho$ of $[0,1]$ such that for every $n\ge1$,
\begin{equation}
\phi[\calH_{[0,\rho n]\times[0,n]}]\ge f(\phi[\calV_{[0,\rho n]\times[0,n]}]),
\end{equation}
where  $\phi$  can be either $\phi^\xi_{\bbG}$ with $\xi\in\{0,1\}$ or $\phi_{S_m}^\xi$ for some $m\ge n$ and $\xi\in\{0,1,0/1\}$.\end{theorem}

Let us make a few concluding remarks on this theorem. 
The choices of $\phi$ are made in such a way that the measures are invariant by translation under the vector $(1,0)$, but finite volume versions of the theorem can be deduced from the same argument. 
 The homeomorphism can be taken to behave like $(x/c\rho^3)^{c\rho}$ for small values of $x$, and $1-c\rho^3(1-x)^{c\rho}$ for values of $x$ close to 1, where $c$ is independent of $\rho$. This fact can be checked by following  the constants carefully through the proof.

\subsection{Renormalization of crossing probabilities}\label{sec:renormalization}

The previous theorem has an important consequence: estimates which are valid for crossing probabilities  in the easy direction can be transferred to estimates on crossing probabilities in the hard direction. 
What the previous theorem does not answer is the possibility that crossing probabilities depend drastically on boundary conditions. In other words, do we have any estimate to start with?

The next theorem states that four different possibilities can happen. Since we know from the RSW theorem that estimates on crossing probabilities can easily be transferred between rectangles, we state the result with the simplest rectangle of all, namely the square box $\Lambda_n$.

\begin{theorem}[Quadrichotomy for crossing probabilities]\label{thm:main}
Fix $p\in[0,1]$ and $q\ge1$. Then, there exists $c>0$ such that one of the following four properties is satisfied:
\begin{description}
\item[(SubCrit)] For every $n\ge1$, $\phi_{\Lambda_{2n}}^1[\calH_{\Lambda_n}]\le \exp(-cn)$;
\item[(SupCrit)] For every $n\ge1$, $\phi_{\Lambda_{2n}}^0[\calH_{\Lambda_n}]\ge 1-\exp(-cn)$;
\item[(ContCrit)] For every $n\ge1$ and every boundary conditions $\xi$, 
$c\le \phi_{\Lambda_{2n}}^\xi[\calH_{\Lambda_n}]\le 1-c;$
\item[(DiscontCrit)] For every $n\ge 1$, 
$\phi_{\Lambda_{2n}}^0[\calH_{\Lambda_n}]\le \exp(-cn)\text{ and }\phi_{\Lambda_{2n}}^1[\calH_{\Lambda_n}]\ge 1-\exp(-cn).$
\end{description}
\end{theorem}

In the first two items, we could have stated the result with ``for all boundary conditions'' since, by comparison between boundary conditions (see Section~\ref{sec:0}), the statement of the theorem implies the estimate for arbitrary boundary conditions. The third property is equivalent to the following statement (which actually follows from the proof of the theorem), which is often referred to as the {\bf box-crossing property}.
\begin{corollary}\label{cor:z}
When {\bf (ContCrit)} occurs, then for every $\rho>0$, there exists $c_\rho>0$ such that for every $n\ge1$ and every boundary conditions $\xi$, 
\begin{equation}\label{eq:zz}c_\rho\le \phi_{[-n,(\rho+1)n]\times[-n,2n]}^\xi[\calH_{[0,\rho n]\times[0,n]}]\le 1-c_\rho.\end{equation}
\end{corollary}

The proof of the theorem relies on a renormalization on so-called strip densities for crossings. The argument is novel and should be very useful in the study of other models. The proof of these results relied on a renormalization scheme using the monotonicity and spatial Markov properties of the model, together with the self-duality that the model enjoys on the square lattice at criticality. The renormalization scheme is both simpler and more robust. In particular, {\em it does not rely on self-duality at criticality}.

We voluntarily mentioned the previous theorem without relation to the phase diagram of the random-cluster model (again, see \cite{Gri06,Dum17a} for details). The motivation comes from future applications which could deal with models without specific parameters. Nonetheless, in our case, the random-cluster model with $q\ge1$ undergoes a phase transition at a critical parameter $p_c(q)$ defined by the property that for every $p\ne p_c(q)$, $\phi^1_{\bbG}=\phi^0_{\bbG}=:\phi_{\bbG}$ and that the $\phi_{\mathbb G}$ probability that there exists an infinite connected component is zero if $p<p_c(q)$ and is one if $p>p_c(q)$. In the course of the proof, we will derive the following corollary.

\begin{corollary}\label{cor:main} The function $q\mapsto p_c(q)$ is continuous on $[1,+\infty)$ and
\begin{itemize}[noitemsep,nolistsep] 
\item if $p<p_c(q)$, then {\bf (SubCrit)} occurs, 
\item if $p>p_c(q)$, then {\bf (SupCrit)} occurs,
\item if $p=p_c(q)$, then  {\bf (ContCrit)} or {\bf (DiscontCrit)} occurs. Furthermore, the set of $q\ge1$ for which {\bf (DiscontCrit)} occurs at $p_c(q)$ is open in $[1,+\infty)$.
\end{itemize}
\end{corollary}
The previous results were proved in the case of the random-cluster model on the square lattice \cite{DumSidTas14} and \cite{DumGanHar16}. Furthermore, it was shown that {\bf (ContCrit)}  occurs at $p=p_c$ when $1\le q\le 4$, while {\bf (DiscontCrit)} occurs when $q>4$. These results were extended to isoradial graphs (which include the triangular and hexagonal lattices) in \cite{DLM}.

\subsection{Applications}

The previous four properties have a number of implications for the model. 
In order to state the applications  properly, let us introduce a number of notions. From now on, $\phi$ always denote a DLR-random-cluster measure on $\mathbb G$.

We say that there is {\em uniqueness of the DLR-random-cluster measure} if $\phi^0=\phi^1$ (in this case, it is known that all DLR-random-cluster measures with the same parameters are equal to $\phi^0$).
The measure $\phi$ is said to satisfy the {\em ratio weak mixing property} with speed $f(k,n)$ if 
 for every $k\le n/2$ and every two events $\mathcal A$  depending on edges in $\Lambda_k$ and $\mathcal B$ depending on edges outside $\Lambda_n$,
 \begin{equation}\label{eq:exp mixing}
  \tag{Mix}|\phi[\mathcal A\cap\mathcal B]-\phi[\mathcal A]\phi[\mathcal B]|\le f(k,n)\phi[\mathcal A]\phi[\mathcal B].\end{equation}
The mixing is said to be {\em polynomial} if $f(k,n)\le (k/n)^c$ for some constant $c\in(0,1)$, and {\em exponential} if $f(k,n)\le \exp(-cn)$. It has been  known for a long time (see e.g.~\cite{Dum17a} and references therein) that $$|\phi[\mathcal A\cap \mathcal B]-\phi[\mathcal A]\phi[\mathcal B]|\le \phi^1_{\Lambda_n\setminus\Lambda_k}[\Lambda_k\longleftrightarrow \partial\Lambda_n]\cdot\phi[\mathcal A]\phi[\mathcal B],$$
so that \eqref{eq:exp mixing} follows from the speed of decay of connectivity probabilities.

The property {\bf (SubCrit)} corresponds to the  typical behavior of a subcritical percolation measure. In this case, there is a unique DLR-random-cluster measure. 
There is no infinite connected component almost surely. The probability that the size of the connected component of $0$ is larger than $n$ decays exponentially fast in $n$ (see Proposition~\ref{cor:2} in Section~\ref{sec:appl-sub-supcrit}).
The measure $\phi$ satisfies the exponential ratio weak mixing (in particular it is ergodic). Furthermore, one can show Ornstein-Zernike asymptotics for the probability that two points are connected (see \cite{CamIofVel08}).  
One can also deduce dynamic properties of the measure, see e.g.~\cite{BlaSin17} for the example of the mixing time of the dynamics of the random-cluster model.

The property {\bf (SupCrit)} corresponds to the typical behavior of a supercritical percolation measure.  In this case, there is a unique DLR-random-cluster measure. 
There exists an infinite connected component almost surely. The probability that  $0$ is connected to a distance $n$ but not to infinity decays exponentially in $n$. Also, the probability that the volume of the connected component of $0$ is of size exactly $n$ is of order $\exp[-O(\sqrt n)]$.
The measure satisfies the exponential ratio weak mixing (in particular it is ergodic). Furthermore, one can show Ornstein-Zernike asymptotics for the probability that two points in a finite connected component are connected (see \cite{CamIofLou10}). As in the previous paragraph, there are consequences for dynamics preserving the measure.

The property {\bf (ContCrit)} corresponds to the typical behavior of a  critical system undergoing a continuous phase transition. In this case, there is a unique DLR-random-cluster measure. 
The measure satisfies the box-crossing property: crossing estimates remain bounded away from 0 and 1 uniformly in the boundary conditions (see Proposition~\ref{cor:2} in the next section).
One deduces that there exists no infinite connected component almost surely. Also, the probability that the size of the connected component of $0$ is larger than $n$ decays faster than $n^{-\alpha}$ and slower than $n^{-\beta}$ for two constants $0<\alpha<\beta<\infty$.
Finally, the measure satisfies the polynomial ratio weak mixing (in particular it is ergodic). 

The property {\bf (DiscontCrit)} corresponds to the typical behavior of a critical system undergoing a discontinuous phase transition. In particular, uniqueness of the DLR-random-cluster measure fails since $\phi^1\ne\phi^0$. Then, there
 exists (resp.~does not exist) an infinite connected component $\phi^1$-almost surely (resp.~$\phi^0$-almost surely). The $\phi^0$-probability that the connected component of $0$ but has radius larger than $n$ decays exponentially in $n$. Also, one can construct non-ergodic DLR-random-cluster measures by taking non-trivial averages of $\phi^0$ and $\phi^1$. There are also consequences for the dynamic aspects (some interesting questions are still open there), see e.g.~\cite{GheLub17}.

The table below summarizes the discussion above. Stars refer to  statements that are not proved in this paper, but should follow from the corresponding results for the random-cluster model. Question marks correspond to open questions.
\begin{center}
\begin{tabular}{|c|c|c|c|c|}
\hline
{\bf Property} & {\bf (SubCrit)} &{\bf (SupCrit)} & {\bf (ContCrit)} & {\bf (DiscontCrit)} \\  \hline
 Regime & $p<p_c$ & $p>p_c$ & $p=p_c$ and $q\le 4$ & $p=p_c$ and $q>4$ \\ \hline
 Existence infinite c.c. & no & yes & no & depends*  \\ \hline
 Ergodicity & yes & yes & yes & no \\ \hline
 uniqueness of Gibbs state & yes & yes & yes & no \\ \hline
 Decay volume of finite c.c. & exponential & stretched-exp & polynomial & exponential \\  \hline
 mixing & exponential & exponential & polynomial & none \\ \hline
 Ornstein-Zernike & yes* & yes* & no* & ?\\ \hline
 \end{tabular}
\end{center}

\subsection{Potential generalizations to other models}\label{sec:other models}

The argument presented in this paper is more general than the one in \cite{DumSidTas14}. Not only does it apply to the random-cluster model on graphs which are not self-dual, but it also finds applications in other models. For these more general applications, it is crucial to get rid of the self-duality argument used in \cite{BefDum12,DumSidTas14}. We therefore want to advertise that the argument presented here is substantially better and more robust.

The first model to come to mind is the random-cluster representation of the dilute Potts model studied in \cite{DumGlaPelSpi17}, which can be thought of as a site percolation version of the random-cluster model. This representation includes an important example, which is the $+/-$ spin representation of the Ising model.

 Even though formally the proof requires strict spatial Markov property (see Section~\ref{sec:0} for a definition), simple modifications can help treat other models having weaker forms of spatial Markov property. In recent years, the super-level lines of random functions have been the object of an intense interest. For instance, logarithmic delocalization of uniformly chosen Lipschitz functions on the hexagonal lattice \cite{GlaMan18} and uniformly chosen homomorphisms on the square lattice \cite{Hom} have been obtained. At the heart of the proof lies a dichotomy theorem that uses ideas developed in this paper. Let us mention that major additional difficulties arise in these models.

\paragraph{Organization} The paper is organized as follows. In Section~\ref{sec:0}, we present some classical facts on the random-cluster model. In Section~\ref{sec:2}, we prove Theorem~\ref{thm:RSW}. In Section~\ref{sec:4}, we discuss crossing probabilities with favorable boundary conditions. In Section~\ref{sec:3}, we show Theorem~\ref{thm:main} and Corollary~\ref{cor:z}. 
In the last section, we prove a few properties following from Theorem~\ref{thm:main} and we prove Corollary~\ref{cor:main}.  

\section{Preliminaries}\label{sec:0}

\subsection{Classical properties of the random-cluster model}

In this section, we list the important properties of the random-cluster model that we will use (they can be found in \cite{Gri06,Dum17a}):
\begin{itemize}
\item (invariance) Let  $\tau$ be an automorphism of $\bbG$. Then for every event $\mathcal A$ depending on the edges in $G$ and every boundary conditions $\xi$, we have
    \begin{equation}
      \label{eq:1}
      \phi_{\tau\cdot G}^{\tau\cdot \xi}[\tau\cdot\mathcal A]= \phi_{G}^{\xi}[\mathcal A],
    \end{equation}
    where $\tau\cdot X$ denotes the image of $X$ under the  action of $\tau$. In particular, $\phi^1_{\mathbb G}$ and $\phi^0_{\mathbb G}$ are invariant under translations.

\item (spatial-Markov property) for every subgragh $H$ (with edge-set $F$) of $G$, every boundary conditions $\xi$ on $G$, and every configuration $\xi'\in\{0,1\}^{E\setminus F}$,
\begin{equation}\label{eq:DMP}\tag{SMP}\phi_G^\xi[\cdot_{|F}|\omega_e=\xi_e \forall e\in E\setminus F]=\phi_H^{\xi\cup\xi'},\end{equation}
where $\xi\cup\xi'$ is the boundary conditions on $H$ given by $x$ and $y$ are in the same element of the partition of $\xi\cup\xi'$ if they are connected in the graph made of the vertices of $G$, where vertices in the same element of the partition $\xi$ are identified, and edge-set given by the open edges of $\xi$. 
\item (FKG inequality) for every increasing events $\mathcal A$ and $\mathcal B$ (an event is increasing if $\omega\le\omega'$ and $\omega$ belonging to this event implies that $\omega'$ also does),
\begin{equation}\label{eq:FKG}\tag{FKG}\phi_G^\xi[\mathcal A\cap \mathcal B]\ge\phi_G^\xi[\mathcal A]\phi_G^\xi[\mathcal B].\end{equation}
\item (comparison between boundary conditions) For every $\zeta$ dominating $\xi$ ($\zeta$ {\em dominates} $\xi$ if every two vertices that are wired in $\xi$ are wired in $\zeta$), and every increasing event $\mathcal A$,
\begin{equation}\label{eq:CBC}\tag{CBC}\phi_G^\xi[\mathcal A]\le \phi_G^\zeta[\mathcal A].\end{equation}
\end{itemize}
We will need a last argument to compare boundary conditions, which is almost tautological: for every two boundary conditions $\xi$ and $\zeta$,
\begin{equation}\label{eq:FI}
\phi^\xi_G[\mathcal A]\le q^{\max\{k_\xi(\omega)-k_{\zeta}(\omega)\,:\,\omega\}-\min\{k_\xi(\omega)-k_{\zeta}(\omega)\,:\,\omega\}}\phi^\zeta_G[\mathcal A].
\end{equation}
We will apply this to bound the ratio of the two probabilities by $q^{k}$, where we change the boundary conditions on $k$ vertices on the boundary. We will also use it in the following special case. Let mix boundary conditions corresponding to two non-trivial partition elements $A$ and $B$, and all the other elements of the partition are singletons. The sets $A$ and $B$ will often be two arcs on the boundary of the graph. We will want to compare these mix boundary conditions to very close ones, called $*$-mix boundary conditions, where the partition is given by $A\cup B$ and singletons. In this case, one obtains
$\phi^{*-\rm mix}_G[\mathcal A]\le \,q\phi^{\rm mix}_G[\mathcal A]$ since there can be only a difference of one between the counts of clusters in both cases. To draw the attention of the reader on the difference between the two boundary conditions, we will try to use the mix boundary conditions consistently  for the case where the two wired arcs are not wired together, and $*$-mix for the one where they are.

  \subsection{Monotonicity in the domain}
The properties \eqref{eq:DMP}, \eqref{eq:FKG} and \eqref{eq:CBC} allow us to compare measures in different domains with suitable boundary conditions. In this section, we describe this monotonicity in the domain, which will be used in many places in the rest of the  paper. Let us begin with a simple and well-known instance of this monotonicity property with wired boundary conditions.  If $G'\subset G$ are two finite subgraphs of $\mathbb G$, then the measure $\phi_G^1$ restricted to $G'$ is stochastically dominated by the measure  $\phi_{G'}^1$, meaning that for every increasing event $\mathcal A$  depending on the edges of $G'$, 
\begin{equation}\label{eq:2}
  \phi_{G}^1[\mathcal A]\le  \phi_{G'}^1[\mathcal A].
\end{equation}
The equation above  is a direct consequence of  \eqref{eq:FKG} and \eqref{eq:DMP}. Indeed, writing $E$ and $E'$ for the edge-sets of $G$ and $G'$ respectively, we have
\begin{equation}\label{eq:3}
  \phi_G^1[\mathcal A] \overset{\eqref{eq:FKG}}{\le}  \phi_G^1[\mathcal A\, | \forall e \in E\setminus E'\ \omega(e)=1] \overset{\eqref{eq:DMP}}=  \phi_{G'}^1[\mathcal A].
\end{equation} 
Using the same idea, we can also compare measures with more general boundary conditions. We formalize this by introducing  a  partial ordering on the set of pairs $(G,\xi)$ where $G$ denotes a finite subgraph of $\mathbb G$ and $\xi$ are boundary conditions on $G$. Write $(G,\xi)\preceq_1(G',\xi')$ if
\begin{equation*}
G'\subset G \quad \text{and}\quad \text{$\xi'$ dominates $\xi\cup 1$}.
\end{equation*}
Intuitively, we have $(G,\xi)\preceq_1 (G',\xi')$ when $(G,\xi)$ is obtained from $(G',\xi')$ by ``pushing'' the wired boundary conditions away. Let us give a first example.

\medskip 
\noindent \textbf{Example 1:} For every boundary conditions $\xi$ on $G$ and $G'\subset G$, we have $(G,\xi)\preceq_1(G',1)$. See Fig.~\ref{fig:5} for an illustration.
\begin{figure}[ht]
  \centering
  \includegraphics[width=5cm]{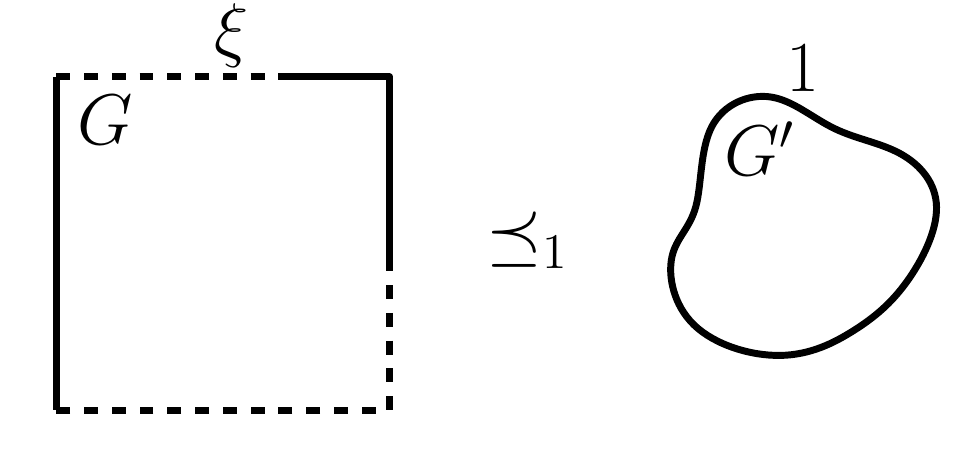}
  \caption{Illustration of Example 1. The solid lines represent wired boundary conditions and the dashed lines represent free boundary conditions. We will often compare domains resulting from explorations of random paths/circuits to a deterministic domain (on the right, a typical domain resulting from the exploration of the outermost circuit and on the left, a deterministic square).}\label{fig:5}
\end{figure}
\medskip

\noindent Of course, a similar statement holds for free boundary conditions. In particular, we can introduce a partial ordering corresponding to ``pushing'' the free boundary conditions away.   Write $(G,\xi)\preceq_0(G',\xi')$ if
\begin{equation*}
G\subset G' \quad \text{and}\quad \text{$\xi$ is dominated by $\xi'\cup 0$}.
\end{equation*} 
Let us give a second example illustrating the two orderings $\preceq_0$ and $\preceq_1$ defined above.

\medskip

\noindent \textbf{Example 2:}
Let $R=[-n,3n]\times[-n,n]$ with boundary conditions $\xi$ defined to be free on $[-n,n]\times\{-n\}$ and $[-n,n]\times\{n\}$ and wired everywhere else. Let $S=[-n,n]^2$ with boundary conditions $\xi'$ defined to be free on top and bottom  and wired everywhere else.  Let  $T=[-n,n]\times[-n,3n]$ with boundary conditions $\xi''$ defined to be wired on $ \{-n\}\times [-n,n]$ and $\{n\} \times [-n,n]$ (the two arcs are also wired together) and free everywhere else. Then, as illustrated on Fig.~\ref{fig:6}, we have
\begin{equation*}
(R,\xi)\preceq_1(S,\xi')\quad\text{ and }\quad (S,\xi')\preceq_0 (T,\xi'').\end{equation*}
\begin{figure}[h]
  \centering
  \begin{minipage}[c]{.7\linewidth} \includegraphics[width=4cm]{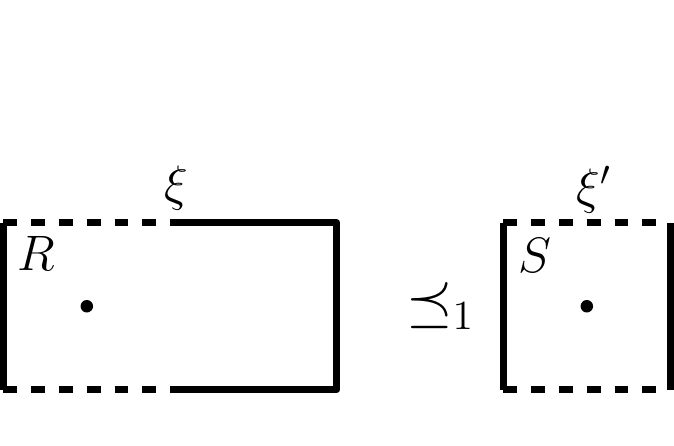}\hfill\includegraphics[width=4cm]{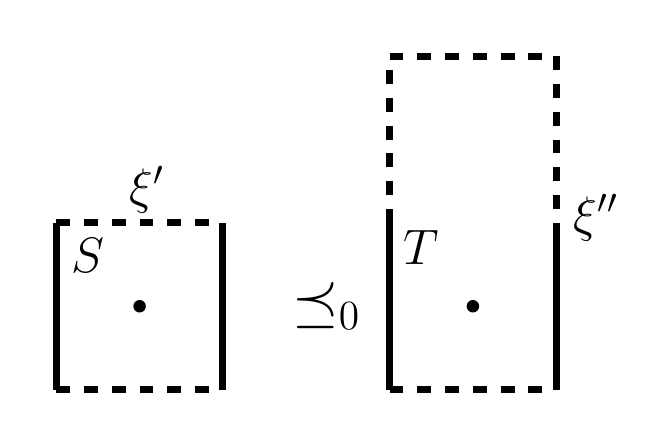}  \end{minipage}
  \caption{Illustration of Example 2. On the left (resp.\@ right) the ordering corresponds to pushing the wired (resp.\@ free)  boundary conditions.}\label{fig:6}
\end{figure}

Using the same reasoning as \eqref{eq:3} and \eqref{eq:CBC}, one can check that the following proposition holds.
\begin{proposition}
  Let $(G,\xi)\preceq_1(G',\xi')$ or $(G,\xi)\preceq_0(G',\xi')$. Then for every increasing event $\mathcal A$ depending on the edges in $G\cap G'$, we have
  \begin{equation}
    \label{eq:4}
    \phi_G^\xi[\mathcal A] \le  \phi_{G'}^{\xi'}[\mathcal A].\tag{MON}
  \end{equation}
\end{proposition}

\begin{remark}
  In applications, we often apply the proposition above twice: first we  push the wired  boundary conditions closer, and then the free boundary away. For instance, in Example 2 above, we have
  \begin{equation*}
    \phi_R^\xi[\mathcal A]\le \phi_T^{\xi''}[\mathcal A]
  \end{equation*}
  for every increasing event $\mathcal A$ depending on the edges in the square $S$.
\end{remark}

\subsection{Strip measures}\label{sec:strip}

In this section, we define random-cluster measures on the infinite strip $S_n=\bbZ\times[-n,2n]$. 

\begin{proposition}\label{prop:1}
  Let $\xi$ be some boundary conditions on $\partial S_n$. For every $m\ge 1$, let $\xi_m$ be boundary conditions on the boundary of $R_{m,n}=[-m,m]\times[-n, 2n]$ inducing the same partition as $\xi$ at the top and bottom sides of $R_{m,n}$ respectively. There exists a  measure $\phi_{S_n}^\xi$ in the strip~$S_n$  characterized by 
  \begin{equation}
    \label{eq:5}
    \phi_{S_n}^{\xi}[\mathcal A]=\lim_{m\to\infty}\phi_{R_{m,n}}^{\xi_m}[\mathcal A]
  \end{equation}
   for every event $\mathcal A$ depending on finitely many edges of $S_n$.
Furthermore, the limit above is independent of the choice of the sequence ($\xi_m$) as long as it induces the same partition as $\xi$ at the top and bottom sides of $R_{m,n}$.
\end{proposition}
The strip measures inherit the properties of $\phi$. Namely, they satisfy \eqref{eq:DMP}, \eqref{eq:CBC}, \eqref{eq:FKG}.
\begin{proof}
Let $\mathcal A$ be an increasing event depending of the edges in $\Lambda:=[-k,k]\times [-n,2n]$. Then, by monotonicity, we can define the increasing limit
 \begin{equation}
   \label{eq:6}
    \phi_{S_n}^\xi[\mathcal A]:=\lim_{\substack{m\to\infty\\ m\ge k}}\phi_{[-m,m]\times[-n,2n]}^{\xi^{(1)}}[\mathcal A],
 \end{equation}
 where $\xi^{(1)}$ are the boundary conditions where all the vertices on the left and right sides of $[-m,m]\times[-n,2n]$ are wired together. 
 By inclusion-exclusion, we extend the measure $\phi_{S_n}^\xi$ to all the events depending on finitely many edges in $S_n$, and by Kolmogorov theorem, to all events.
 
 Now, on the strip, one can easily check using finite-energy that $\Lambda$ is connected to the left or right side of $[-m,m]\times[-n,2n]$ with probability tending to 0 as $m$ tends to infinity. Thus, 
 $$\lim_{\substack{m\to\infty\\ m\ge k}}\phi_{[-m,m]\times[-n,2n]}^{\xi^{(1)}}[\mathcal A]=\lim_{\substack{m\to\infty\\ m\ge k}}\phi_{[-m,m]\times[-n,2n]}^{\xi^{(0)}}[\mathcal A],$$
 where $\xi^{(0)}$ is the boundary conditions on $[-m,m]\times[-n,2n]$ corresponding to $\xi$ except that no two vertices on the left or right sides are wired together.
  The convergence \eqref{eq:5} follows readily by \eqref{eq:CBC}. 
\end{proof}

From now on, write $\phi_{S_n}^0$, $\phi_{S_n}^1$ and $\phi_{S_n}^{0/1}$ for the measures corresponding to the boundary conditions $\xi$ defined to be respectively always equal to 0, always equal to 1, and equal to 1 if and only if $e\subset \bbR\times(-\infty,0]$. It follows from the previous proposition and the hypotheses on $\phi$ that these three measures are invariant w.r.t.~to horizontal translations and vertical reflections. 

\subsection{Duality}
\label{sec:duality}
Define $\bbG^*$  to be the dual  graph of $\bbG$, obtained by putting a vertex in every face of $\bbG$, and a dual edge $e^*$ between vertices corresponding to faces bordered by the same edge $e$ of $\bbG$. When $G$ is a finite subgraph of $\bbG^*$, let $G^*$ be the subgraph of $\bbG^*$ with vertex set given by the edges $e^*$ with $e\in E$, and vertices being the endpoints of these vertices. 
Then, for a measure $\phi$ on $G$, define a dual measure $\phi^*$ on $G^*$ as follows: $e^*$ is open in $\phi^*$ if $e$ is closed in $\phi$, and vice versa. 

The only information that we will need is that the dual of $\phi^1_{\Lambda_n}$ (resp.~$\phi^0_{\Lambda_n}$) is a random-cluster measure on the dual graph with free boundary conditions (resp.~wired). In fact, one can check that the parameters $(p^*,q^*)$ of the dual measure are given by the equations
$$\frac{pp^*}{(1-p)(1-p^*)}=q\qquad\text{and}\qquad q^*=q.$$
Furthermore, the relation extends to infinite volume.  The dual of $\phi^1$ and $\phi^0$ are DLR-random-cluster measures on $\bbG^*$ with free and wired boundary conditions respectively. Equivalently, dual measures can be defined in the strip in a natural way. One can define the dual strip $S_n^*$ as subgraphs of $\mathbb G^*$ and the dual of $\phi_{S_n}^0,\phi_{S_n}^1$ and $\phi_{S_n}^{0/1}$ are $\phi_{S_n^*}^0,\phi_{S_n}^1$ and $\phi_{S_n}^{1/0}$.

Note that the property {\bf (SupCrit)} is dual to {\bf (SubCrit)} in the sense that the model satisfies one if and only if its dual satisfies the other one. On the contrary {\bf (ContCrit)} and {\bf (DiscontCrit)} are self-dual: the model satisfies one if and only if its dual does. 

\section{Russo-Seymour-Welsh theory}\label{sec:2}

 In this whole section, $\phi$ can be either $\phi^\xi_{\bbG}$ with $\xi\in\{0,1\}$ or $\phi_{S_m}^\xi$ for some $m\ge n$ and $\xi\in\{0,1,0/1\}$. The goal in this section is to establish Theorem~\ref{thm:RSW} and the key step in the proof will be the following proposition.

\begin{proposition}\label{prop:q} For any $\rho>0$, there exists $c_0=c_0(\rho)>0$ such that for every $n\ge 1/\rho$, 
\begin{equation}\label{eq:prop:q}\phi[\calH_R]\ge c_0\phi[\calV_R]^{1/c_0},\end{equation}
where $R:=[0,\rho n]\times[0,n]$. \end{proposition}

The proof is presented next section but before that, we show how the proposition implies the theorem. This part of the proof is classical.
\begin{proof}[Proof of Theorem~\ref{thm:RSW} (using Proposition~\ref{prop:q})]
The inequality \eqref{eq:prop:q} provided by Proposition~\ref{prop:q} is useful for values of $\phi[\calV_R]$ which are close to $0$, but proving the existence of the homeomorphism $f$ for every $x\in[0,1]$ requires to check that if $\phi[\calV_R]$ is close to 1, then so is $\phi[\calH_R]$. 

In order to do that, we apply the proposition to the dual measure $\phi^*$ of $\phi$ (see the definition in  Section~\ref{sec:duality}). If $\phi=\phi_{S_n}^\xi$ is a strip measure, one can check that its dual version $\phi^*$ corresponds to a strip measure for the dual model, hence Proposition~\ref{prop:q}  applies and
$$\phi^*[\calH_R]\ge c_0\phi^*[\calV_R]^{1/c_0}.$$
Yet, $\phi^*[\calH_R]=1-\phi[\calV_R]$ and $\phi^*[\calV_R]=1-\phi[\calH_R]$ give that 
$$1-\phi[\calV_R]\ge c_0(1-\phi[\calH_R])^{1/c_0},$$
which in turns implies that 
\begin{equation}\label{eq:prop2}\phi[\calH_R]\ge 1-c_0^{-c_0}\times(1-\phi[\calV_R])^{c_0}.\end{equation}
The existence of $f$ follows readily from \eqref{eq:prop:q} and \eqref{eq:prop2}.
\end{proof}

To conclude the proof of Theorem~\ref{thm:RSW}, we therefore need to show Proposition~\ref{prop:q}. 
Set $k=\lceil n/50\rceil$ and introduce the rectangle $R_0:=[-17k,18k]\times[0,n]$ and the horizontal segment $S_0:=[0,k]\times\{0\}$ centered on its bottom (the constants 17, 18 and 50 are there for convenience but any constants $a,b,c$ larger than those and satisfying $b=a+1$ and $c\ge 2a+5$ would do). We also introduce the translates $R_j$ and $S_j$ of these sets by the vector $(jk,0)$.

In the RSW theory, the difficulty comes from the fact that vertical crossings of wide rectangles are not much constrained, so that it is difficult to combine them into crossings staying in some chosen area (for instance to create horizontal crossings of very long rectangles). It will therefore not come as a surprise that the heart of the proof is encapsulated in the following lemma, which shows that the probability of having a ``bridge'' between segments at the bottom of the rectangle, of size $k$ and separated by a segment of size $k$ can be bounded from below in terms of the probability of crossing the rectangle vertically. Different crossings bridging between different segments can then easily be combined to create long horizontal segments. Let us formalize this. Consider the events (see Fig.~\ref{fig:2})
\begin{equation}
\calA_j:=\big\{S_j\lr[R_j\cup R_{j+4}] (S_{j+2}\cup S_{j+4})\big\}.\label{eq:7}
\end{equation}

\begin{figure}\begin{center}
\includegraphics[width=0.70\textwidth]{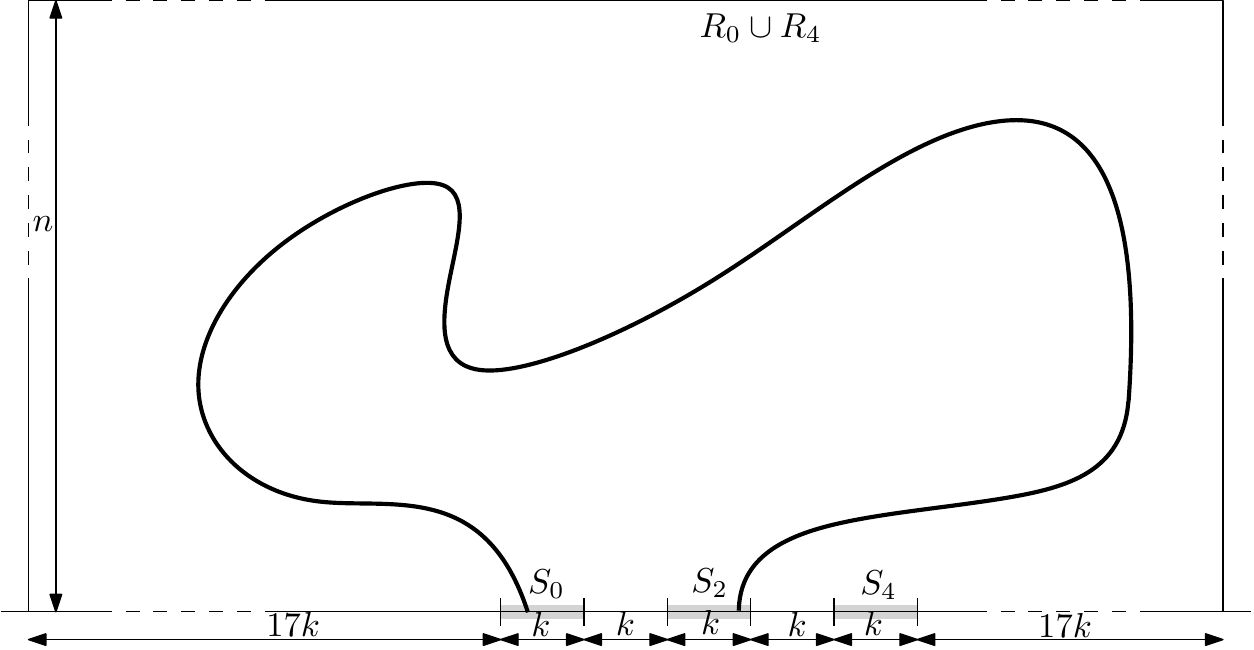}
\caption{\label{fig:2}The event $A_0$, the segments $S_0$, $S_2$, $S_4$ as well as the rectangle $R_0\cup R_4$. Despite what may appear in the picture, the height $n$ of the rectangle $R_0\cup R_4$ is larger than its width $\frac{39}{50}n$.}\end{center}
\end{figure}

\begin{lemma}\label{lem:RSW}
There exists a constant $c_1>0$ such that, for every $\lambda>0$ and every integer $n$,
$$\phi[\calA_0]\ge \frac{c_1}{\lambda^3}\phi[\calV_{[0,\lambda n]\times[0,n]}]^3.$$
\end{lemma}

Before proving this statement, let us conclude the proof of the proposition. 
 If the events $\calA_j$
occur for every $-1\le j\le 50\rho n$, then $[0,\rho n]\times[0,n]$ is crossed horizontally by an open path. Therefore, using the invariance under translation and the previous lemma in the last inequality, we obtain that 
\begin{equation}\label{eq:8}\phi[\calH_{[0,\rho n]\times[0,n]}]\ge \phi\big[\bigcap_{j=-1}^{50\rho n}\calA_j\big]\stackrel{\eqref{eq:FKG}}\ge \prod_{j=-1}^{50\rho n}\phi[\calA_j]\ge \Big(\frac{c_1}{\lambda^3}\phi[\calV_{[0,\lambda n]\times[0,n]}]^3\Big)^{50\rho n+2}.\end{equation}
The theorem follows by taking $\lambda =\rho$. 

As mentioned above, combining bridges to create long crossings is a standard fact. The real difficulty of the theorem remains hidden in the proof of Lemma~\ref{lem:RSW} below. Before diving into the proof, let us explain the strategy. 

Assume that the segments $S_0$, $S_2$ and $S_4$ are all connected to the top of the rectangle. We would like to show that with good probability, two of these segments are connected. 
The idea will be to show that the left-most vertical crossing $\Gamma_1$ starting from $S_0$ and the right-most vertical crossing $\Gamma_2$ starting from $S_4$ can be used to create a symmetric domain, i.e.~a domain that enjoys some rotation or reflection symmetry.  Then, we will show that, conditioned on $\Gamma_1$ and $\Gamma_2$, the symmetric domain is ``bridged'' by an open path with good probability.
The idea of using symmetric domains goes back to \cite{BefDum12}, and was later used in \cite{DumSidTas14}. In previous works, estimates on crossing probabilities were obtained using the self-duality of the model. Instead of using self-duality (which is unavailable here), we used that conditioned on $\Gamma_1$ and $\Gamma_2$,  the vertical crossing from $S_2$ to the top must in particular cross the symmetric domain, and that it must do it in such a way that its connected component does not intersect $\Gamma_1$ or $\Gamma_2$. It is possible to use \eqref{eq:CBC} to prove that the probability of this event is in fact smaller than the conditional probability that the symmetric domain is bridged by an open path. In conclusion,  we replace the estimate obtained by duality by an estimate obtained thanks to the existence of this other path from $S_2$ to the top. 

\begin{proof}[Proof of Lemma~\ref{lem:RSW}] By increasing $\lambda$ if needed, we may assume that $\lambda\ge1$. Set $R:=[0,\lambda n]\times[0,n]$ and $C:=\lambda n/k$. Define the events (see Fig.~\ref{fig:3})
\begin{align*}
\calT_j&:=\{S_j\lr[R_j]{\rm T}_{R_j}\},&\\
\calL_j&:=\{S_j\lr[R_{j-13}]{\rm L}_{R_{j+4}}\}&&\text{and}\quad\qquad\ \calR_j:=\{S_j\lr[R_{j+13}]{\rm R}_{R_{j-4}}\},\\
\calL_j'&:=\{S_j\lr[R_j]{\rm L}_{R_j}\}\setminus\calL_j&&\text{and }\quad\qquad\calR_j':=\{S_j\lr[R_j]{\rm R}_{R_j}\}\setminus\calR_j.\end{align*}
For $\calV_{R}$ to occur, one of the segments $S_j$ for $0\le j<C$ must be connected either to the left, top or right of $R_j$ (in $R$). If it is to the left, then $\omega$ is either in $\calL_j$ or $\calL'_j$, and similarly for the right. The union bound implies that
\begin{equation}\label{eq:p}\max\{\phi[\calT_j],\phi[\calL_j],\phi[\calR_j],\phi[\calL_j'],\phi[\calR_j']:0\le j<C\}\ge \frac{\phi[\calV_R]}{6C}.\end{equation}
Thus, translation and reflection invariances give that
\begin{equation}
  \label{eq:9} \max\{\phi[\calT_0],\phi[\calL_0],\phi[\calL_0']\}\ge \frac{\phi[\calV_R]}{6C}.
\end{equation}
\begin{figure}\begin{center}
\includegraphics[width=0.365\textwidth]{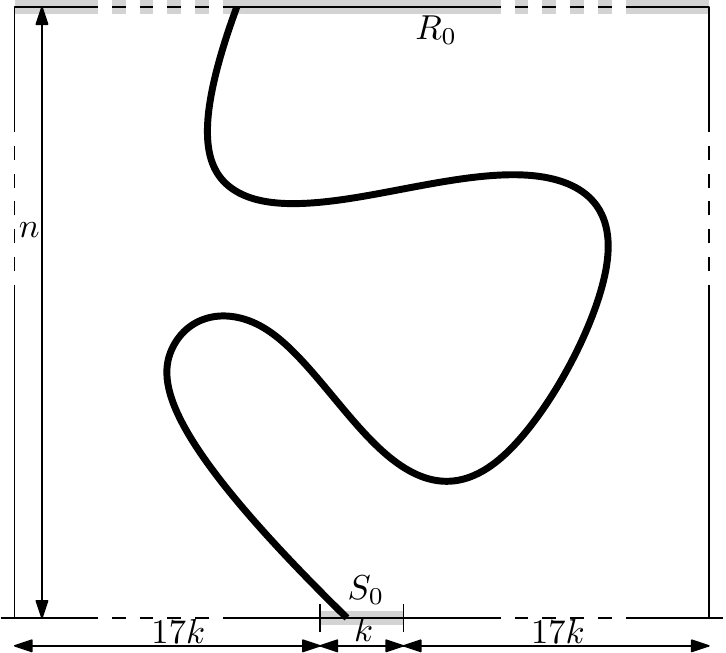}\includegraphics[width=0.29\textwidth]{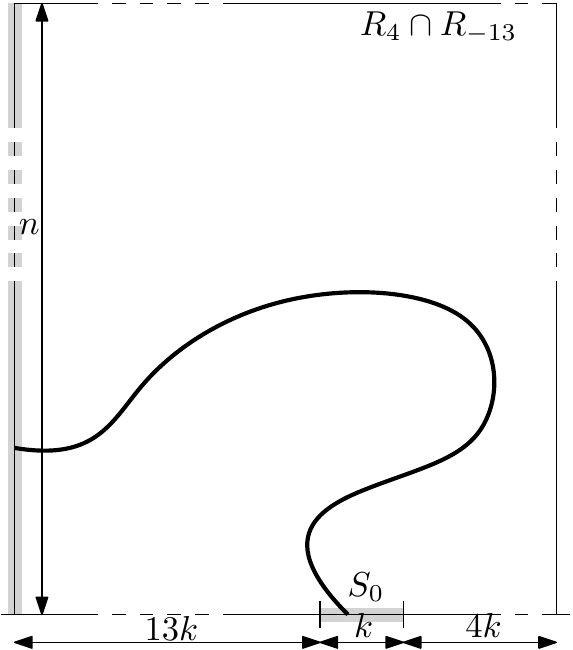}\includegraphics[width=0.365\textwidth]{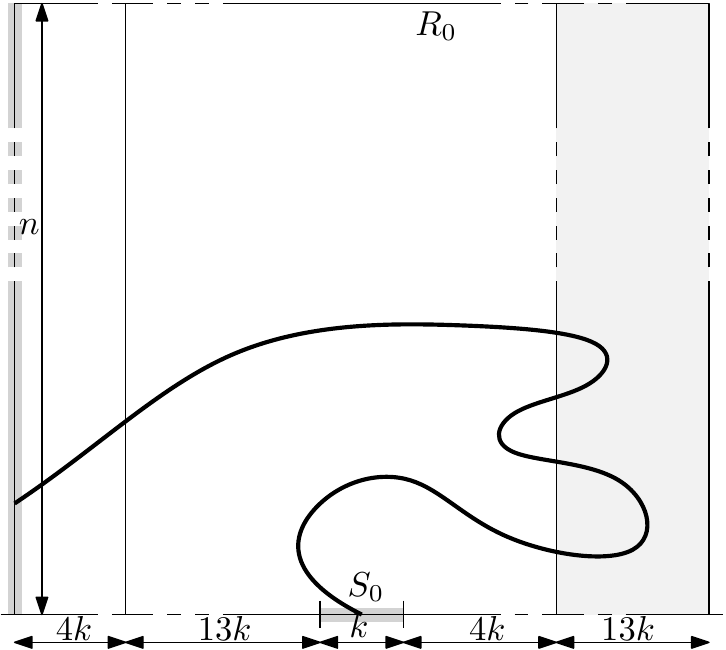}
\caption{\label{fig:3}The three events $\calT_0$, $\calL_0$ and $\calL'_0$.}\end{center}
\end{figure}

The proof is easy to conclude when $\phi[\calL_0]\ge \phi[\calV_R]/6C$. Indeed, in this case, by reflection and invariance under translations, we have that $\phi[\calR_0]$ and $\phi[\calL_4]$ are larger than or equal to ${\phi[\calV_R]}/{6C}$. Thus,
\begin{equation}
  \label{eq:10}
  \phi[\mathcal A_0]\ge \phi[ \calR_0\cap \calL_4]\stackrel{\eqref{eq:FKG}}\ge \Big(\frac{\phi[\calV_R]}{6C}\Big)^2.
\end{equation}
Therefore, for the rest of the proof, we can assume that
\begin{equation}
\max\{\phi[\calT_0],\phi[\calL_0']\}\ge \frac{\phi[\calV_R]}{6C}.
\end{equation}
Assume for a moment that we proved that for  $\calC$ equal to  $\calT$ or  $\calL'$, 
\begin{equation}\label{eq:kk}\phi[\calA_0|\calC_0\cap\calC_4]\ge \tfrac1q\cdot\phi[\calC_2\setminus(\calA_0\cup\calA_2) |\calC_0\cap\calC_4].\end{equation}
Thus, $\phi[\calA_0]=\phi[\calA_2]$ implies that
\begin{align*}(q+2)\phi[\calA_0]&\stackrel{\phantom{\eqref{eq:kk}}}\ge q\,\phi[\calA_0\cap\calC_0\cap\calC_4]+\phi[(\calA_0\cup\calA_2)\cap\calC_0\cap\calC_4]\\
                                &\stackrel{\eqref{eq:kk}}\ge \phi[\calC_0\cap\calC_2\cap\calC_4]\stackrel{\eqref{eq:FKG}}\ge \phi[\calC_0]^3.\end{align*}
                              Combined with \eqref{eq:10}, this implies the lemma with $c_1:=(6C)^{-3}/(q+2)$>0. 
                                                            To finish the proof completely, we now prove \eqref{eq:kk}. The construction is different depending on whether $\calC$ is equal to $\calT$ or $\calL'$.

\paragraph{Proof of \eqref{eq:kk} with $\calC=\calT$.} 
For $\omega\in\calT_0\cap\calT_4$, let $\Gamma_1(\omega)$ be the left-most open path in $R_0$ from $S_0$ to ${\rm T}_{R_0}$ and $\Gamma_2(\omega)$ be the right-most open path in $R_4$ from $S_4$ to ${\rm T}_{R_4}$. 
It is sufficient to show that for every $\gamma_1$ and $\gamma_2$ such that  $\{\Gamma_1=\gamma_1,\Gamma_2=\gamma_2\}\subset\calT_0\cap\calT_4$, 
\begin{equation}\label{eq:ak}\phi[\calA_0|\Gamma_1=\gamma_1,\Gamma_2=\gamma_2]\ge \tfrac1q\cdot\phi[\calT_2\setminus(\calA_0\cup\calA_2) |\Gamma_1=\gamma_1,\Gamma_2=\gamma_2].\end{equation}
Consider the square ${\rm Sym}:=[-17k,22k]\times[0,39k]$ and let $\Omega$ be the points in ${\rm Sym}$ that are between $\gamma_1$ and $\gamma_2$, i.e.~on the right of $\gamma_1$ and the left of $\gamma_2$. 

If $\gamma_1$ is connected to $\gamma_2$ in $\Omega$, then $S_0$ is connected to $S_4$ in $R_0\cup R_4$, and in particular $\calA_0$ is satisfied. Hence
\begin{align*}
  \phi[\calA_0\:|\:\Gamma_1=\gamma_1,\Gamma_2=\gamma_2]&\ge\phi[\gamma_1\longleftrightarrow\gamma_2\text{ in }\Omega\: |\: \Gamma_1=\gamma_1,\Gamma_2=\gamma_2].
\end{align*}
Observe that conditioned on $\Gamma_1=\gamma_1$ and $\Gamma_2=\gamma_2$, the boundary conditions on $\Omega$ are  dominating the boundary conditions $\xi$ with vertices of $\gamma_1$ wired together, and vertices  $\gamma_2$ wired together. Therefore, by \eqref{eq:DMP} and \mon, we have
\begin{equation*}
  \phi[\gamma_1\longleftrightarrow\gamma_2\text{ in }\Omega\: |\: \Gamma_1=\gamma_1,\Gamma_2=\gamma_2]\ge \phi_\Omega^{\xi}[\gamma_1\longleftrightarrow\gamma_2]\ge\phi_{\rm Sym}^{\rm mix}[\gamma_1\longleftrightarrow\gamma_2\text{ in }\Omega],
\end{equation*}
where the ${\rm mix}$ boundary conditions are wired on ${\rm L}_{\rm Sym}$, wired on the ${\rm R}_{\rm Sym}$, and free everywhere else. The two equations above give
\begin{equation}
  \label{eq:11}
  \phi[\calA_0\:|\:\Gamma_1=\gamma_1,\Gamma_2=\gamma_2]\ge\phi_{\rm Sym}^{\rm mix}[\gamma_1\longleftrightarrow\gamma_2\text{ in }\Omega]\ge\phi_{\rm Sym}^{\rm mix}[{\rm L}_{\rm Sym}\longleftrightarrow{\rm R}_{\rm Sym}] .
\end{equation}
\medbreak
On the other hand, for $\omega\in \calT_2\setminus(\calA_0\cup\calA_2)$, $\omega$ must contain a crossing of $\Omega$ from bottom to top included in a connected component in $\Omega$ that does not touch $\gamma_1$ or $\gamma_2$. Calling this  event $\calE$ and using \eqref{eq:DMP} and \mon, we find that
 $$\phi[\calT_2\setminus(\calA_0\cup\calA_2) |\Gamma_1=\gamma_1,\Gamma_2=\gamma_2]\le \phi^{\xi'}_\Omega[\calE],$$
 where the boundary conditions $\xi'$ are 0 on $\gamma_1$ and $\gamma_2$, and 1 everywhere else. Then using \eqref{eq:DMP}, \mon~and symmetries, we get
 \begin{align*}\phi^{\xi'}_\Omega[\calE]\le \phi_{\rm Sym}^{\rm mix'}[\calE]\le \phi_{\rm Sym}^{\rm mix'}[{\rm T}_{\rm Sym}\longleftrightarrow{\rm B}_{\rm Sym}]= \phi_{\rm Sym}^{*-\rm mix}[{\rm L}_{\rm Sym}\longleftrightarrow{\rm R}_{\rm Sym}],\end{align*}
where the $*-{\rm mix}$ boundary conditions are wired on ${\rm L}_{\rm Sym}\cup{\rm R}_{\rm Sym}$, and free elsewhere, and ${\rm mix}'$ is the rotation of $*-{\rm mix}$ by $\pi/2$.
In conclusion
\begin{equation}\label{eq:o}\phi[\calT_2\setminus(\calA_0\cup\calA_2) |\Gamma_1=\gamma_1,\Gamma_2=\gamma_2]\le \phi_{\rm Sym}^{*-\rm mix}[{\rm L}_{\rm Sym}\longleftrightarrow{\rm R}_{\rm Sym}].\end{equation}
Now, \eqref{eq:FI} implies that the probabilities with ${\rm mix}$ and $*-{\rm mix}$ boundary conditions are related by a factor at most $q$. Hence, \eqref{eq:11} and \eqref{eq:o} imply \eqref{eq:ak}, and therefore \eqref{eq:kk} by averaging over every $\gamma_1,\gamma_2$.

\paragraph{Proof of \eqref{eq:p} with $\calC=\calL'$.}The proof is based on the same idea, except that the construction of $\Omega$ and ${\rm Sym}$ is slightly more complicated (in particular ${\rm Sym}={\rm Sym}(\Omega)$ will depend on $\Omega$).  Let $\omega\in\calL'_0\cap\calL'_4$. Let $\Gamma_1(\omega)$ be the left-most open path in $R_0$ from $S_0$ to ${\rm L}_{R_4}$ and $\Gamma_2(\omega)$ be the right-most open path in $R_4$ from $S_4$ to ${\rm L}_{R_4}$. We wish to prove the equivalent of \eqref{eq:ak} and therefore fix $\gamma_1$ and $\gamma_2$.

Introduce the vertical line $\ell=\{5k\}\times\bbR$ and consider the following paths:
\begin{itemize}[noitemsep]
\item Let $\gamma'_1$ be the part of $\gamma_1$ going from $S_0$ to the first intersection $x$ with $\ell$;
\item Let $\gamma'_2$ be the part of $\gamma_2$ bordering the connected component of $x$ in $\bbH\setminus \gamma_2$, where  $\bbH$ is the half-plane on the right of $\ell$;
\item Let $\tilde \gamma_1$ be the part of the reflection (with respect to $\ell$) of $\gamma'_1$ going from $\ell$ to the first intersection with $\gamma'_2$;
\item Let $\tilde\gamma_2$ be the part of the reflection of $\gamma'_2$ going from $\ell$ to the first intersection with $\gamma'_1$;
\end{itemize}
The assumption that $\omega\in \calL'_0\cap\calL'_4$ guarantees that the point $x$ exists and that the paths intersect. Let ${\rm Sym}(\Omega)$ be everything enclosed in $\gamma'_1\cup\tilde\gamma_1\cup\gamma'_2\cup\tilde\gamma_2$. Let $\Omega$ be the subdomain of ${\rm Sym}(\Omega)$ made of points between $\gamma_1$ and $\gamma_2$; see Fig.~\ref{fig:4}. 

Now, the proof runs as before. The boundary conditions induced by $\Gamma_1=\gamma_1$ and $\Gamma_2=\gamma_2$ on $\Omega$ dominate 
the boundary conditions $\xi$ equal wired on $\gamma_1$ and wired on $\gamma_2$, and free on the rest of the boundary of $\Omega$.  Also, if $\gamma'_1$ is connected to $\gamma'_2$ in ${\rm Sym}(\Omega)$, then $\gamma_1$ is connected to $\gamma_2$ in $\Omega$. Thus, exactly in the same way as we obtained \eqref{eq:11}, we find
\begin{equation*}\phi[\calA_0|\Gamma_1=\gamma_1,\Gamma_2=\gamma_2]\ge\phi^{\rm mix}_{{\rm Sym}(\Omega)}[\gamma_1'\longleftrightarrow\gamma_2'],\end{equation*}
where the {\rm mix} boundary conditions are wired on $\gamma'_1$, wired on $\gamma'_2$, and free everywhere else.

Also, observe that for $\calL_2\setminus(\calA_0\cup\calA_2) $ to occur, $\Omega$ must contain an open path from $\tilde\gamma_1$ to $\tilde\gamma_2$ which is not connected to $\gamma_1$ or $\gamma_2$ in $\Omega$. Exactly as we obtained \eqref{eq:o}, we find \begin{align*}\phi[\calL_2\setminus(\calA_0\cup\calA_2) |\Gamma_1=\gamma_1,\Gamma_2=\gamma_2]& \le \phi^{*-\rm mix}_{{\rm Sym}(\Omega)}[\gamma_1'\longleftrightarrow\gamma'_2],\end{align*}
where the reader will easily deduce from the previous case the definition of $*-{\rm mix}$. The end of the proof is the same.
\end{proof}
\begin{figure}[htbp]\begin{center}
\includegraphics[width=0.463\textwidth]{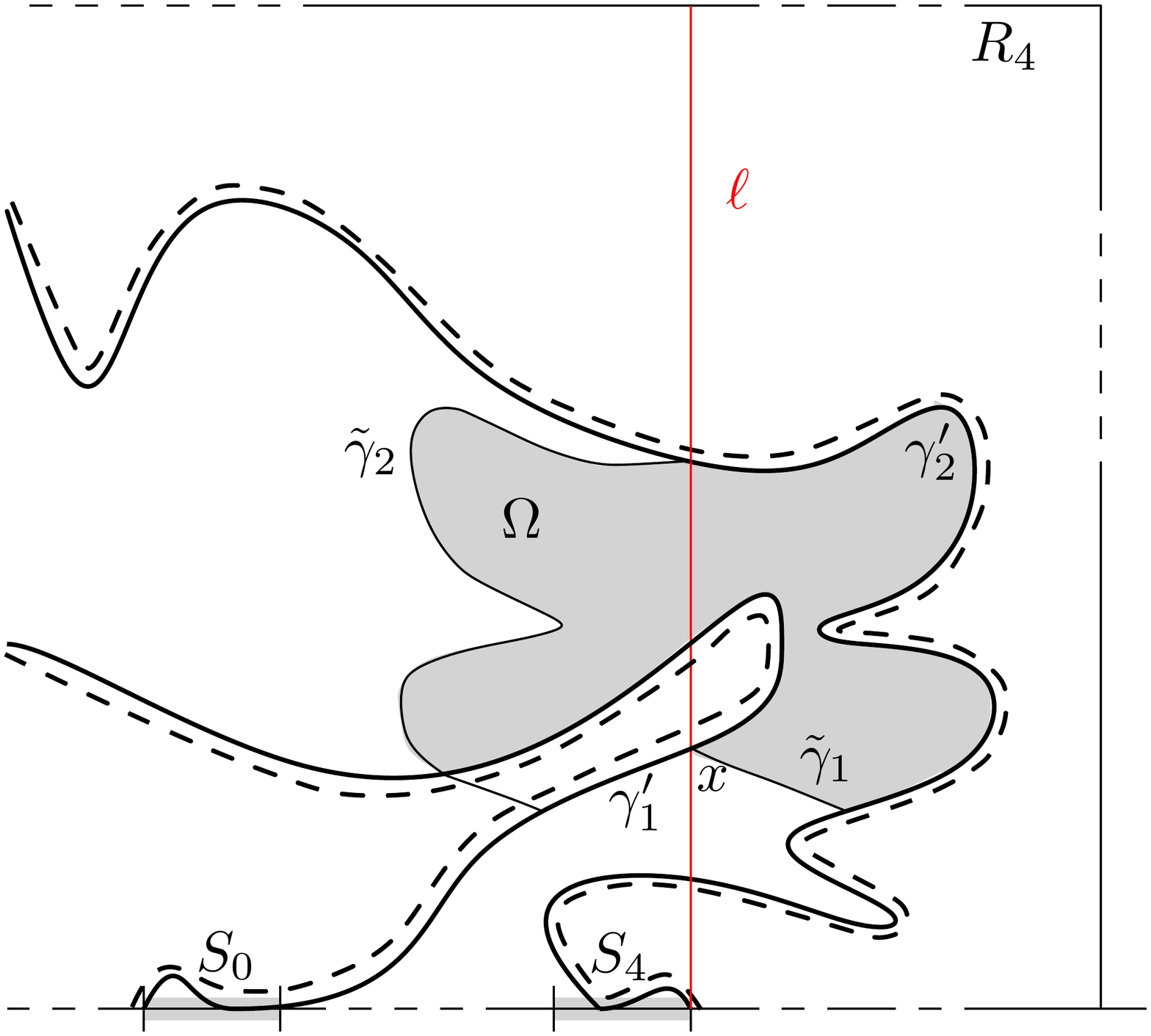}
\caption{\label{fig:4}The construction of $\Omega$ in the second case.}\end{center}
\end{figure}

\section{Crossing probabilities with wired boundary conditions}\label{sec:4}

In this section, we use coarse graining ideas to control crossing probabilities with wired boundary conditions.

\begin{lemma}\label{lem:a}
For every $C\ge 2$, there exists $\delta>0$ such that, if $\phi_{\Lambda_{Ck}}^1[\Lambda_k\longleftrightarrow\partial\Lambda_{2k}]<\delta$ for some $k$, then there exists $c>0$ such that  for every $n,N\ge2$ and every $x\in \Lambda_n$,
$$\phi_{\Lambda_{n}}^1[\text{the connected component of $x$ in $\Lambda_{n-Ck}$ has volume $N$}]\le \exp(-cN).$$
\end{lemma}

\begin{proof}Fix a constant $\mu<\infty$ such that the number of connected sets of size $\ell$ containing the origin in $\bbZ^2$ is smaller than $\mu^\ell$. (The existence of $\mu$ is a standard fact in the study of ``animals'' in a graph, see e.g.~\cite{Gri99a}).     
If the connected component of $x$ is of size $N$, there exists a connected set $S$ of $N/|\Lambda_k|$ vertices in $\bbZ^2$ containing $x$ such that for every $y\in S$, the box of size $k$ around $y$ is connected to the boundary of the box of size $2k$ around $y$. One may choose a subset of $4^{-C}N/|\Lambda_k|$ vertices of $S$ which are at a distance $2Ck$ of each other, so that the union bound, \eqref{eq:DMP} and  \eqref{eq:CBC} imply that
$$\phi_{\Lambda_{n}}^1[\text{the connected component of $x$ in $\Lambda_{n-Ck}$ has volume $N$}]\le (\mu \delta^{4^{-C}})^{N/|\Lambda_k|}.$$
The proof follows by choosing $\delta$ small enough.\end{proof}

This lemma shows that if {\bf non(SubCrit)}, then $\phi_{\Lambda_{Ck}}^1[\calA_k]\ge \delta$ for every $k$. It is simple to apply the argument of the previous section to show that crossing probabilities of a rectangle of size $n\times\rho n$ in a box of size $Cn$ with boundary conditions 1 do not tend to zero. In fact, we wish to show a slightly stronger result dealing with probability measures in strips. This corollary will be instrumental in the proof of Theorem~\ref{thm:main}. 
\begin{corollary}\label{cor:1}
For every $\rho>0$ and $\lambda\ge 1$, there exists $c_3=c_3(\lambda,\rho)>0$ such that \begin{itemize}
\item if {\bf non(SubCrit)}, then for every $n$, $\displaystyle\phi^1_{S_{\lambda n}}[\calH_{[0,\rho n]\times[0,n]}]\ge c_3,$
\item if {\bf non(SupCrit)}, then for every $n$, $\phi^0_{S_{\lambda n}}[\calV_{[0,\rho n]\times[0,n]}]\le 1-c_3.$
\end{itemize}
Furthermore, for $\rho,\lambda \ge 2$, we can choose $c_3=c_3(\lambda,\rho)=\lambda^{-C\rho}$ where $C>0$ is a constant independent of $\rho$ and $\lambda$. 
\end{corollary}
This corollary illustrates the fact that the difficulty in Section~\ref{sec:3} will be to show that either the probabilities of crossing with free (resp.~wired) boundary conditions go exponentially fast to 0 (resp.~1), or that they remain uniformly bounded away from 0 and 1. Indeed, with wired boundary conditions, crossing probabilities remain bounded away from  zero, while with free, they remain bounded away from 1. At the risk of repeating ourselves after what we wrote in the introduction, Theorem~\ref{thm:main} is really dealing with the impact of boundary conditions.

The strategy of the proof is the following (see Fig~\ref{fig:R_i}). We will show that, in a strip of fixed height $n$ with wired boundary conditions on bottom and top, one can ``bring'' wired boundary conditions from $-\infty$ and $\infty$ to a distance $8n$ by asking for the existence of vertical crossings in rectangles of width $2^in$ for $i\ge3$. Proceeding from infinity guarantees that, at each step, previous crossings induce wired boundary conditions at a small distance. This enables us to use the estimate on $\phi_{\Lambda_{Ck}}^1[\calA_k]$ to show that the probability of having crossings (say) at step $i$ knowing those at step $i+1$ is of order $1-\exp(-c/2^i)$, and that therefore the probability of having all of them is bounded away from 0 uniformly in $n$.

\begin{figure}[htbp]\begin{center}
\includegraphics[width=1.00\textwidth]{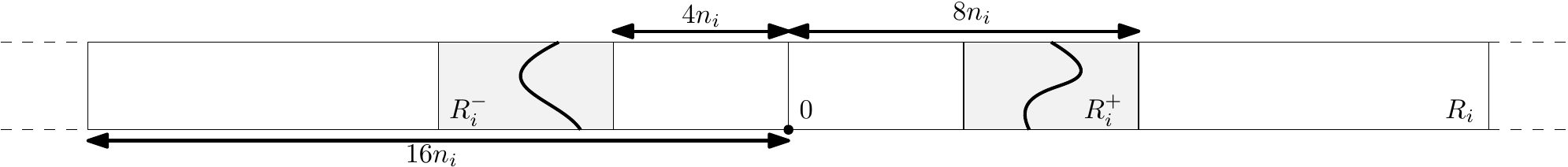}
\caption{\label{fig:R_i}
The rectangles $R_i$, $R_i^-$ and $R_i^+$. Conditionally on the vertical crossings $R_i^-$ and $R_i^+$, induced boundary conditions on $R_{i-1}$ dominate wired boundary conditions on the boundary of $R_{i-1}$.
}\end{center}
\end{figure}

\begin{proof}
  The second statement follows from the first one by duality so we focus on the first one. By monotonicity, one can assume that $\rho\ge 1$. The previous lemma combined with {\bf non(SubCrit)} implies the existence of $C>0$ such that for every $k\ge1$,
$$\phi_{\Lambda_{22k}}^1[\calA_{k}]\ge 4 e^{-C}.$$
If the event $\calA_{k}$ occurs, then one of the four rotated versions of $\mathcal V_{[-2k,2k]\times[k,2k]}$ (we consider the four rotations with center 0 and angle $j\tfrac \pi 2$, $j=0,1,2,3$) must occur. Hence, by symmetry and the union bound, we deduce that\begin{equation}
  \label{eq:12}
\phi_{[-16k,16k]\times[0,k]}^1[\mathcal V_{[4k,8k]\times[0,k]}]\stackrel{\mon}\ge  \phi_{\Lambda_{22k}}^1[\mathcal V_{[-2k,2k]\times[k,2k]}]\ge e^{-C}. 
\end{equation}
For $i\ge 0$ and $n$, set $n_i=2^in$ and define the rectangles (see Fig.~\ref{fig:R_i})
\begin{align*}
R_i&:=[-16n_{i},16n_{i}]\times[0,n],\\
R_i^-&:=[-8n_{i},-4n_{i}]\times[0,n],\\
R_i^+&:=[4n_{i},8n_{i}]\times[0,n],
\end{align*}
and let $\mathcal E_i$ be the event that $R_i^-$ and $R_i^+$ are crossed vertically. Thus, using in the second inequality that a vertical crossing of $[4n_{i},8n_{i}]\times[0,n_i]$ crosses vertically $2^i$ translates of $[4n_i,8n_i]\times[0,n]$, we deduce that \begin{align}
 \phi_{R_i}^1[\mathcal E_i]\stackrel{\eqref{eq:FKG}}\ge\phi_{R_i}^1[\mathcal V_{R_i^+}]^2\stackrel{\eqref{eq:4}}\ge \phi_{[-16n_i,16n_i]\times[0,n_i]}^1[\mathcal V_{[4n_{i},8n_{i}]\times[0,n_i]}]^{2^{1-i}}\stackrel{\eqref{eq:12}}\ge e^{-C2^{1-i}}.\label{eq:14}
\end{align}
Fix $j\ge 3$. Consider the rectangle $R_{j}$ and assume that the event $\mathcal E_i$ occurs for some $i\in \{1,\ldots,j\}$. Conditioning on the left-most vertical crossing in $R_i^-$ and the right-most vertical crossing in $R_i^+$ and considering the boundary conditions induced on the area $\Omega$ between them, \eqref{eq:DMP} and \eqref{eq:4} imply 
\begin{equation}
  \label{eq:15}
  \phi_{R_j}^1[\mathcal E_{i-1}\:|\:\mathcal E_i]\ge  \phi_{R_{i-1}}^1[\mathcal E_{i-1}] \ge  e^{-C 2^{1-i}}.
\end{equation}
Therefore,
\begin{equation}
\phi_{R_j}^1[\mathcal E_0]\ge \phi^1_{R_j}[\mathcal E_j] \cdot\prod_{1\le i\le j}  \phi_{R_j}^1[\mathcal E_{i-1}\:|\:\mathcal E_i] \ge  e^{-4C}.\label{eq:16}
\end{equation}
Letting $j$ tend to infinity, we obtain that
\begin{equation}
  \label{eq:17}
  \phi_{S_n}^1[\mathcal V_{[4n,8n]\times[0,n]}]\ge \phi_{S_n}^1[\mathcal E_0]\ge e^{-4C}.
\end{equation}
Applying the inequality above to $\lambda n$ and then Theorem~\ref{thm:RSW} -- more precisely \eqref{eq:8} if one wants the bound on $c_3$ -- to translate this estimate on the probability of $[4\lambda n,8\lambda n]\times[0,\lambda n]$ being crossed vertically into the probability that the rectangle $[0,\rho n]\times[0,n]$ is crossed horizontally concludes the proof.
\end{proof}

\section{Proof of Theorem~\ref{thm:main} and Corollary~\ref{cor:z}}\label{sec:3}

The proof of Theorem~\ref{thm:main} will be based on a renormalization involving the following {\em strip densities}:
 \begin{align}
    \label{eq:18}
    p_n&:=\limsup_{\alpha\to\infty} \left(\phi_{[0,\alpha n]\times[-n,2n]}^0[\mathcal H_{[0,\alpha n]\times[0,n]}]\right)^{1/\alpha},\\
    q_n&:=\limsup_{\alpha\to\infty} \left(\phi_{[0,\alpha n]\times[-n,2n]}^1[\mathcal V_{[0,\alpha n]\times[0,n]}^c]\right)^{1/\alpha},
\end{align}
where $\calV^c_R$ denotes the complement of the event $\calV_R$. The quantity $p_n$ provides information on the linear cost of a long open path in the strip. The quantity $q_n$ is its dual analogue. Even though we will not be using those facts, let us note that one may prove that the limsup is in fact a true limit. Also, it follows from~\mon~that for every $n$ and $\lambda\ge1$,  $p_{\lambda n}\ge p_n^\lambda$ and $q_{\lambda n}\ge q_n^\lambda$.

The proof of the theorem will be divided into four parts. In the next section, we explain how to relate $p_n$ to $q_n$. In Section~\ref{sec:pushing}, we derive a lemma, called the pushing lemma, which will be fundamental in the reminder of the proof. Section~\ref{sec:renor} contains the proof of a recursive inequality between $p_{3n}$ and $p_n$. This inequality implies that either $p_n$ does not decay at all, or it decays exponentially fast. The last section wraps up the proof of Theorem~\ref{thm:main}. 

\begin{figure}[htbp]\begin{center}
\includegraphics[width=0.45\textwidth]{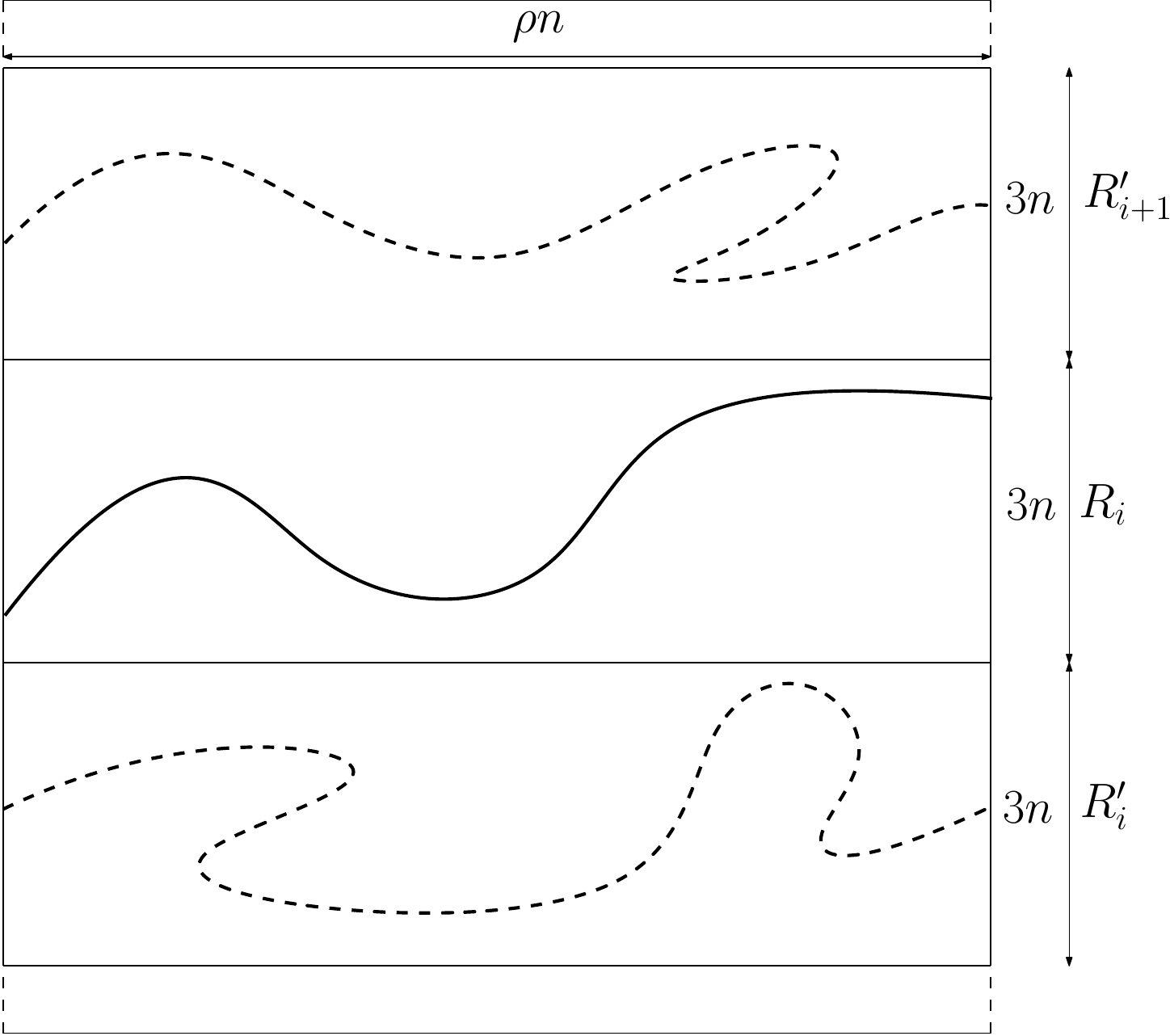}
\caption{\label{fig:lem_2}The rectangles $R_i$, $R'_i$ and $R'_{i+1}$. Also, we depicted the events involved in $\mathcal E$ and $\mathcal F$, namely the existence of a horizontal open crossing in each $R_i$, and the non-existence of vertical crossings in each $R_{i}'$ (here, in $R_i'$ and $R_{i+1}'$), which we depicted by their dual picture, meaning a dash path referring to the existence of a dual open path.}\end{center}
\end{figure}

\subsection{Relation between $p_n$ and $q_n$}
\label{sec:cross-dens-infin}

Notice that $p_n\simeq q_n$ when the system is self-dual like for instance the square lattice (this is not exactly an equality because the dual graph is slightly translated compared to the primal graph). The next lemma shows that crossing densities $(p_n)$ and $(q_n)$ have the same behavior at criticality even for more general systems not enjoying self-duality.

\begin{lemma}\label{lem:1}
  Assume \textbf{non(SubCrit)} and \textbf{non(SupCrit)}. Then, there exists a constant $C>0$ such that for every integer $\lambda\ge 2$ and every $n\in3 \mathbb N$,
\begin{equation}
  \label{eq:20}
   p_{3n}\ge \frac 1{\lambda^C} \ q_{n}^{3+3/\lambda} \quad\text{and}\quad q_{3n}\ge \frac 1{\lambda^C} \ p_n^{3+3/\lambda}.  
\end{equation} 
\end{lemma}

The proof consists of four steps. First, we estimate the probability in a strip with wired boundary conditions of the event $\mathcal E$ that some rectangles $R_i$ of height $3n$ and width $\alpha n$, vertically spaces by rectangles $R'_i$ of height $3n$ (see Fig.~\ref{fig:lem_2}),  are horizontally crossed. To do this, we use Corollary~\ref{cor:1} and the fact that boundary conditions are wired. Second, consider the event $\mathcal F$ that none of the rectangles $R'_i$ in between the previous rectangles $R_i$ are vertically crossed. In order to estimate the probability of $\mathcal F$ conditioned on $\mathcal E$, one uses the probability that a rectangle of height $n$ is not crossed vertically when there are wired boundary conditions at a distance $n$ of the rectangle. This type of probability is involved in the definition of $q_{n}$. Third, using \eqref{eq:FI} to impose free boundary conditions on the left and right of the rectangles $R_i$ and $R'_i$ -- this event is denoted by $\mathcal G$ -- paying an exponential cost on the probability which involves $n$ but not $\alpha$. Finally, one estimates the probability of the intersection $\mathcal E\cap\mathcal F\cap\mathcal G$ using that conditionally on $\mathcal F\cap\mathcal G$, the boundary conditions in each rectangle $R_i$ are dominated by the boundary conditions induced by free boundary conditions at a distance $3n$ by \eqref{eq:DMP} and \mon. As a consequence, the probability of a horizontal crossing involved in the definition of $q_n$ can be bounded in terms of crossing probabilities involved in the definition of $p_{3n}$. Overall, letting $\alpha$ go to infinity at fixed $n$ implies an inequality between $q_n$ and $p_{3n}$. The other inequality can be obtained by duality. 

\begin{proof}
  Fix $\lambda\in\bbN$ and $n\in 3\mathbb N$. We prove the first inequality of \eqref{eq:20} only since the second inequality follows from the same reasoning by duality.
  Let $\alpha\gg1$ be a large number such that $\alpha n$ is an integer (the reader should keep in mind that $\alpha$ will tend to infinity at the end of the proof). 
For every $0\le i\le\lambda$, define the rectangles (see Fig.~\ref{fig:lem_2})
 \begin{align*}
 R&:=[0,\alpha n]\times[0,6\lambda n+3n],\\
 R_i&:=[0,\alpha n]\times[6in+3n,6in+6n],\\
  R'_i&:=[0,\alpha n]\times[6in,6in+3n]. 
 \end{align*}
 Let $\calE$ be the event that each $R_i$ with $0\le i\le \lambda-1$ is crossed horizontally. Using Corollary~\ref{cor:1} and~\mon~in the last inequality, we find that 
  \begin{equation}
    \label{eq:21}
       \phi_{R}^1[\mathcal E]\stackrel{\eqref{eq:FKG}}\ge \prod_{0\le i\le \lambda-1}\phi_R^1[\calH_{R_i}]\ge \Big(\frac 1 {\lambda^C} \Big)^{\lambda\alpha}.
  \end{equation}
  Let $\calF $ be the event that none of the rectangles $R'_i$ with $0\le i\le \lambda$ is crossed vertically. 
Since the event $\mathcal E$ depends only on edges outside of the union of the $R'_i$, \eqref{eq:DMP}, \eqref{eq:4} and the inclusion of events give that
  \begin{equation}
    \label{eq:22}
    \phi_R^1[\mathcal F |\mathcal E]\ge \prod_{0\le i\le \lambda}\phi_{R'_i}^1[\calV^c_{R'_i}]\ge\phi^1_{[0,\alpha n]\times[-n,2n]}[\mathcal V_{[0,\alpha n]\times[0,n]}^c]^{\,\lambda+1}.
  \end{equation}
By \eqref{eq:FI}, we have that  
\begin{equation}
  \label{eq:23}
  \phi_R^1[\mathcal E\cap \mathcal F\cap \calG]\ge q^{-(12\lambda+6) n} \phi_R^1[\mathcal E\cap \mathcal F].
\end{equation}
  Now,  we provide an upper bound on $\phi_R^1[\mathcal E\cap \mathcal F\cap\mathcal G]$. 
  Using \eqref{eq:DMP} and \eqref{eq:4} several times implies \begin{equation}
    \label{eq:24}
  \phi_R^1[\mathcal E\cap \mathcal F\cap\mathcal G] \le \phi_R^1[\mathcal E| \mathcal F\cap\mathcal G]\le \phi_{[0,\alpha n]\times [-3n,6n]}^0[\mathcal H_{[0,\alpha n]\times[0,3n]}]^{\,\lambda}.
  \end{equation}
Combining the four previous displayed equations gives
  \begin{equation}
    \label{eq:25}
\phi_{[0,\alpha n]\times [-3n,6n]}^0[\mathcal H_{[0,\alpha n]\times[0,3n]}]^{\lambda}\ge \frac {q^{-(12\lambda+6) n}}{\lambda^{C\lambda\alpha}} \ \phi^1_{[0,\alpha n]\times[-n,2n]}[\mathcal V_{[0,\alpha n]\times[0,n]}^c] ^{\lambda+1}.
  \end{equation}
Taking both sides to the power $1/(\lambda\alpha)$ and then letting $\alpha$ tend to infinity leads to
  \begin{equation}
    \label{eq:26}
    p_{3n}^{1/3}\ge \frac 1{\lambda^C} \ q_{n}^{1+1/\lambda}. 
  \end{equation}
\end{proof}

\subsection{The pushing lemma}\label{sec:pushing}

We will use the previous result through the following lemma that we pompously named the pushing lemma since it will enable us to ``push'' free boundary conditions later on in the next section.
\begin{lemma}[Pushing Lemma]\label{lem:2} There exists $c>0$ such that for every $n\ge1$, we have either
\begin{equation}
  \label{eq:27}
 \forall \alpha\ge1, \quad \phi^{1/0}_{\overline R}[\calH_R]\ge c^\alpha, \tag{\text{PushPrimal}}  
  \end{equation}
or
\begin{equation}
  \forall \alpha\ge1, \quad  \phi^{0/1}_{\overline R}[\calV_R^c]\ge c^\alpha,\tag{\text{PushDual}}\label{eq:28}
\end{equation}
where $R:=[0,\alpha n]\times[0,n]$, $\overline R:=[0,\alpha n]\times[0,26n]$ \text{ and }1/0 (resp.~0/1) refers to the wired boundary conditions  on the union of the left, top and right sides of $\overline R$ and free elsewhere (resp.~wired on the bottom and free elsewhere).
\end{lemma}

In the following lemma, we show that in the strip $S_n:=\mathbb Z\times[-n,2n]$ with free boundary conditions on top and wired on bottom, we can either create long horizontal open crossings in $S_n$, or that this statement is true for the dual measure. 
\begin{lemma}
There exists a constant  $c>0$ such that for every $n\ge 1$, either
\begin{equation}
  \label{eq:29}
 \forall \alpha\ge1, \quad\phi_{S_n}^{0/1}[\mathcal H_{[0,\alpha n]\times[0,n]}]\ge c^\alpha,  
\end{equation}
or
\begin{equation}
 \label{eq:30}
  \forall \alpha\ge1, \quad\phi_{S_n}^{0/1}[\mathcal V_{[0,\alpha n]\times[0,n]}^c]\ge c^\alpha.
\end{equation}
\end{lemma}

\begin{figure}\begin{center}
\includegraphics[width=0.65\textwidth]{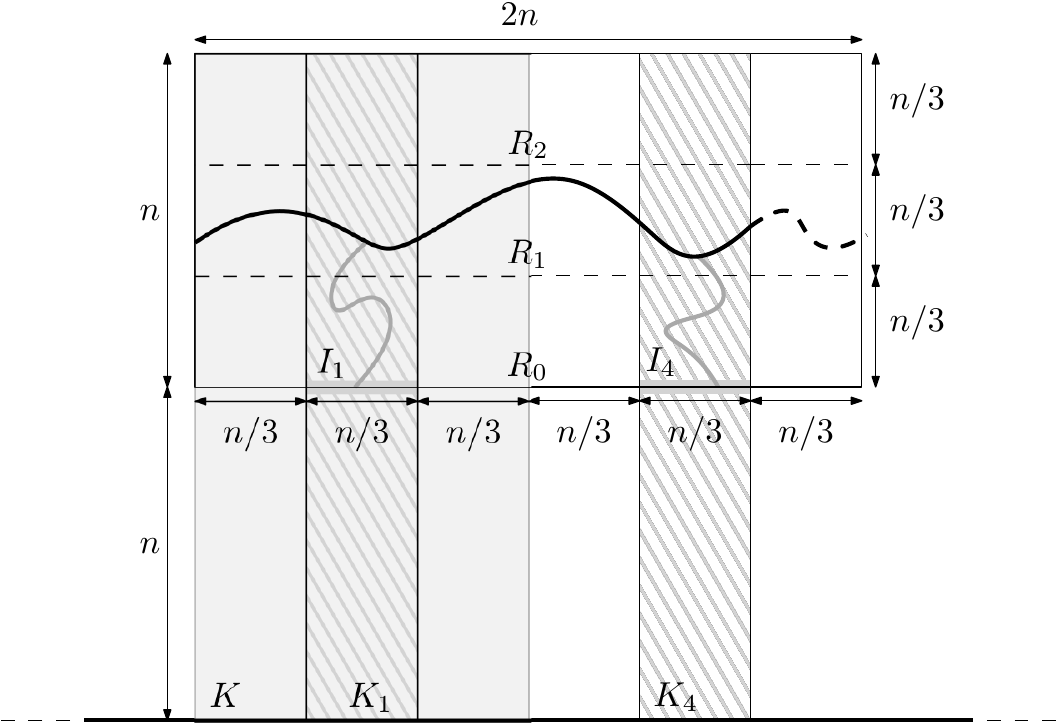}
\caption{\label{fig:lem-3}The rectangles $R_i$ and  $K_i$, as well as the horizontal segments $I_1$ and $I_4$. The rectangle $K=K_0\cup K_1\cup K_2$ is denoted in light gray. The path $\Gamma$ and wired boundary conditions  at the bottom induce boundary conditions that are dominating the mix$'$ boundary conditions on the parts of $K_1$ and $K_4$ that is below $\Gamma$.}\end{center}
\end{figure}
\begin{proof}
  Consider the three rectangles 
  $$R_i:=[0,2n]\times[\tfrac i3n,\tfrac{i+1}3n]\quad\text{for } i=0,1,2.$$ If $\phi_{S_n}^{0/1}[\mathcal V_{R_i}]\ge \frac16$  (resp.~$\phi_{S_n}^{0/1}[\mathcal H_{R_i}^c]\ge \frac16$) for some $i$, then  Theorem~\ref{thm:RSW} applied to $\phi_{S_n}^{0/1}$ (resp.~its dual)  concludes that \eqref{eq:29}  (resp.~\eqref{eq:30}) holds. Therefore, we may assume that for every $i$,  
  \begin{equation}
\phi_{S_n}^{0/1}[\mathcal V_{R_i}^c]\ge \tfrac56 \quad\text{ and }\quad\phi_{S_n}^{0/1}[\mathcal H_{R_i}]\ge \tfrac 56\label{eq:31}.
\end{equation}
Using the union bound, together with \eqref{eq:DMP} and~\mon~in the second inequality, we find that 
  \begin{equation}
    \label{eq:32}
    \tfrac12\le  \phi_{S_n}^{0/1}[\mathcal V_{R_0}^c\cap \mathcal H_{R_1}\cap \mathcal V_{R_2}^c]\le \phi_{R}^{*-\rm mix}[\mathcal H_{R_1}],  
  \end{equation}
  where $R:=[0,2n]\times[0,n]$ and the $*-\rm mix$ boundary conditions are wired on the union of the left and right sides, and free elsewhere. 
  
For $i=0,\dots,5$, introduce the horizontal segments and the rectangles (see Fig.~\ref{fig:lem-3})\begin{align*}
I_i&:=[\tfrac i3n,\tfrac{i+1}3n]\times\{0\},\\
 K_i&:=[\tfrac i3n,\tfrac{i+1}3n]\times[-n,n].\end{align*}  We claim that
\begin{equation}
  \label{eq:33}
  \phi_{S_n}^{0/1}[I_1\longleftrightarrow I_4\text{ in $R$}\: |\: \mathcal H_{R_1}]\ge\tfrac1{64}.
\end{equation}
Indeed, condition on  the top-most horizontal crossing $\Gamma$  of $R_1$. Properties \eqref{eq:DMP} and \eqref{eq:4} give that the boundary conditions induced by $\Gamma$ on the part of $K_1$ below $\Gamma$ are dominating the boundary conditions induced by the mix boundary conditions equal to wired on bottom, wired on top parts of $K:=K_0\cup K_1\cup K_2$, and free everywhere else. Using the rotational symmetry (to compare $K$ to $R$) and comparing mix and $*$-mix boundary conditions, we deduce that   
\begin{equation}
  \label{eq:34}
  \phi_{S_n}^{0/1}[\Gamma \lr[] I_1\text{ in $K_1$}|\Gamma]\ge \tfrac1q\,\phi_{R}^{\rm mix}[\mathcal H_{R_1}] \stackrel{\eqref{eq:32}}\ge\tfrac1{8q}.
\end{equation}
Since the same holds true for $I_4$, \eqref{eq:FKG} implies that
\begin{equation}
  \label{eq:35}
  \phi_{S_n}^{0/1}[I_1 \lr[] I_4\text{ in $R$}|\Gamma]\ge\tfrac{1}{64q^2},  
\end{equation}
which gives \eqref{eq:33} by integrating on $\Gamma$. 

In conclusion, \eqref{eq:31} and \eqref{eq:33} give that
\begin{equation}
  \label{eq:36}
   \phi_{S_n}^{0/1}[I_1\lr[]I_4\text{ in $R$}]\ge  \phi_{S_n}^{0/1}[I_1\lr[]I_4\text{ in $R$} |\mathcal H_{R_1}]\phi_{S_n}^{0/1} [\mathcal H_{R_1}]\ge \tfrac{5}{384q^2},
\end{equation}
which can be combined with \eqref{eq:FKG} as in the proof of \eqref{eq:8} to get \eqref{eq:29}. \end{proof}

\begin{proof}[Proof of Lemma~\ref{lem:2}]
  Without loss of generality, we may assume $n\in 3\mathbb N$: the cases $n=1,2$ can be treated using finite energy, and the other cases can be obtained using \eqref{eq:4}.
We assume that \eqref{eq:30} holds for $n/3$ and prove that \eqref{eq:28} holds for $n$. If \eqref{eq:29} holds instead of \eqref{eq:30}, the same argument proves \eqref{eq:27}.  

For $i=1,\ldots,78$, consider the rectangles 
$$R_i:=[0,\alpha n]\times [\tfrac i3n,\tfrac{i+1}3n].$$ 
By \eqref{eq:4} and \eqref{eq:29} applied to $n/3$, we get
\begin{equation*}
    \phi^{0/1}_{\overline R}[\mathcal V_{R_{78}}^c]\ge c^\alpha.
\end{equation*}
Then, by conditioning on the top-most dual path in $R_{i+1}$ and using \eqref{eq:4}, we find that for every $1\le i<78$, 
\begin{equation*}
  \phi^{0/1}_{\overline R}[\mathcal V_{R_i}^c\: |\: \mathcal V_{R_{i+1}}^c ]\ge c^\alpha,
\end{equation*}
The two equations above imply that  $\phi^{0/1}_{\overline R}[\mathcal V_{R_1}^c]\ge c^{78\alpha}$ and the proof is complete.
\end{proof}

\subsection{The renormalization inequality}\label{sec:renor}
The main motivation for introducing crossing densities is the renormalization inequality in the following lemma.

\begin{lemma}\label{lem:3}
If \textbf{non(SubCrit)} and  \textbf{non(SupCrit)}, then there exists a constant $C>0$ such that for every integer $\lambda\ge2$ and every $n\in 3\mathbb N$, we have 
  \begin{equation}
    \label{eq:37}
    p_{3n}\le \lambda^C\, p_n^{3-9/\lambda}\quad and \quad q_{3n}\le \lambda^C\, q_n^{3-9/\lambda}.
  \end{equation}
\end{lemma}

The proof is very similar to the proof of the second inequality of Lemma~\ref{lem:1}, except that we use the pushing lemma to bring back the boundary conditions induced by the occurrence of the event $\mathcal F$ closer to the rectangles $R_i$. In order to salvage as much notation from the proof of Lemma~\ref{lem:1} as possible, we consider $n\in 9\mathbb N$ and prove the inequality for $n/3$. Also, we modify slightly the definition of $\mathcal E$ by forcing the horizontal crossings to occur in rectangles of height $n/9$ instead of $3n$. These two modifications allow us to replace $p_{3n}^{1/3}$ by $p_{n/9}^9$. Finally, using Lemma~\ref{lem:1} enables us to replace $p_{n/9}^9$ by $q_{n/3}^3$, a fact which concludes the proof.
\begin{proof}Assume $n\in 9\mathbb N$. 
We focus on the second inequality only (the first one follows by applying the same reasoning to the dual). By Lemma~\ref{lem:2}, either \eqref{eq:28} or \eqref{eq:27} holds true. Assume that it is \eqref{eq:28} that holds. We explain at the end of the proof how to modify the proof if it is \eqref{eq:27} that holds. Fix $\lambda,\alpha\ge 1$. We use again the rectangle  $R$ and the events  $\mathcal F$ and $\mathcal G$ defined in the proof of Lemma~\ref{lem:1}. Divide, for every $0\le i\le \lambda-1$, the middle of $R_i$ into three thinner rectangles
 \begin{align*}
  \tilde R^-_i&:=[0,\alpha n]\times[6in+3n+\tfrac{12}9n,6in+3n+\tfrac{13}9n],\\
   \tilde R_i&:=[0,\alpha n]\times[6in+3n+\tfrac{13}9n,6in+3n+\tfrac{14}9n],\\
    \tilde R^+_i&:=[0,\alpha n]\times[6in+3n+\tfrac{14}9n,6in+3n+\tfrac{15}9n]. 
 \end{align*}
Let $\tilde{\mathcal E}$ be the event that every rectangle $\tilde R_i$ is crossed horizontally. The same steps as in the proof of Lemma~\ref{lem:1} lead to
\begin{equation}
  \label{eq:38}
  \phi_{R}^1[\tilde{\mathcal E}\cap \mathcal F\cap \mathcal G]\ge \frac {q^{-(12\lambda+6) n}} {\lambda^{C\lambda\alpha}}\phi^1_{[0,\alpha n]\times[-n,2n]}[\mathcal V_{[0,\alpha n]\times[0,n]}^c]^{\lambda+1}.
\end{equation}
We now use \eqref{eq:28}. Let $\tilde{\mathcal F}$ be the event that none of the rectangles $\tilde R^\pm_i$ are crossed vertically.
By the same reasoning that we have already used several times (conditioning on lowest/highest dual crossings and using \eqref{eq:4}) and the assumption that  \eqref{eq:28} holds for $n/9$, we find
\begin{equation*}
  \phi_{R}^1[\tilde{\mathcal F}\:|\:\tilde{\mathcal E}\cap \mathcal F\cap \mathcal G] \ge  c^{2\lambda\alpha}.
\end{equation*}
Therefore,
\begin{equation}
  \label{eq:39}
  \phi_R^1[\tilde{\mathcal E}\cap \tilde{\mathcal F}\cap \mathcal G]= \phi_{R}^1[\tilde{\mathcal F}|\tilde{\mathcal E}\cap \mathcal F\cap \mathcal G] \phi_{R}^1[\tilde{\mathcal E}\cap \mathcal F\cap \mathcal G]\ge  c^{2\lambda\alpha} \phi_{R}^1[\tilde{\mathcal E}\cap \mathcal F\cap \mathcal G].  
\end{equation}
Using the same reasoning as in \eqref{eq:24}, we can also prove
\begin{equation}
  \label{eq:40}
 \phi_R^1[\tilde{\mathcal E}\cap \tilde{\mathcal F}\cap \mathcal G]\le \phi^0_{[0,\alpha n]\times[-n/9,2n/9]}[\mathcal H_{[0,\alpha n]\times[0,n/9]}]^\lambda.
\end{equation}
Combining \eqref{eq:38}, \eqref{eq:39} and \eqref{eq:40} and letting $\alpha$ go to infinity, we deduce that, by possibly increasing $C$,  
\begin{equation}
  \label{eq:42}
  p_{n/9}^9\ge \frac{1}{\lambda^C}\ q_{n}^{1+1/\lambda}.
\end{equation}
The second inequality of Lemma~\ref{lem:1} applied to $n/9$ implies that
\begin{equation}
  \label{eq:43}
  q_{n/3}^3\ge \frac 1{\lambda^{(4+1/\lambda)C}} \ q_{n}^{(1+1/\lambda)^2}.
\end{equation}
This concludes the proof when \eqref{eq:28} holds for $n/9$. If on the contrary \eqref{eq:27} holds for $n/3$, it also does for $n$ by \eqref{eq:4} (with a potentially larger constant $C$) so that we may establish
\begin{equation}
  \label{eq:44}
  q_{n/3}^9\ge  \frac 1{\lambda^C} \ p_{3n}^{1+1/\lambda}
\end{equation}
in the same way we proved \eqref{eq:42} (with the appropriate modifications of the rectangles, and working with the dual picture). The first inequality of Lemma~\ref{lem:1}  applied to $n$ implies that \eqref{eq:43} holds in this case as well. 
\end{proof}

\subsection{Proof of Theorem~\ref{thm:main} and Corollary~\ref{cor:z}}
\label{sec:proof_Thm}

We assume \textbf{non(SubCrit)} and  \textbf{non(SupCrit)} and  prove that either \textbf{(ContCrit)} or \textbf{(DiscontCrit)} occur. In fact, the proof of \textbf{(ContCrit)} will also imply Corollary~\ref{cor:z}.

Lemma~\ref{lem:3} implies that, along the geometric subsequence $n=3^i$, we have either
    \begin{itemize}
    \item[(i)] $p_n\le \exp(-c n)$ for every $n$ (for some constant $c>0$ independent of $n$), or
    \item[(ii)] $\inf p_n>0$ for every $n\ge1$.
    \end{itemize}
    To see this, apply first \eqref{eq:37} to (say) $\lambda=20$. This shows that either $p_n$ is uniformly positive, or it decays stretch exponentially fast. Then, apply \eqref{eq:37} with $\lambda=n^2$ to strengthen  the stretched exponential decay into an exponential one.
\bigbreak
In order to conclude the proof, we show that (i) implies  \textbf{(DiscontCrit)} and (ii) implies \textbf{(ContCrit)}.
 
Assume that (i) holds. If $\Lambda_n$ is crossed horizontally and $L$ denotes the length of the largest edge in $\mathbb G$, there exist two vertices in $x\in[-n,-n+L]\times[-n,n]$ and $y\in [n-L,n]\times[-n,n]$ respectively that are connected to each others inside $\Lambda_n$. By the union bound and quasi-transitivity, we have that 
\begin{equation}
  \phi_{\Lambda_{2n}}^0[\mathcal H_{\Lambda_n}]\le cn^2  \phi_{\Lambda_{2n}}^0[x\lr y \text{ in $\Lambda_n$}].
\end{equation}
where $c>0$ is a constant. 
 By quasi-transitivity and finite energy, we can further assume that  $x$ and $y$ are in fact in $\{-n\}\times[-n,n]$ and $\{n\}\times[-n,n]$.
For every integer $k$, let $x_k:=x+(4kn,0)$. Then, using \eqref{eq:4}, symmetry by reflection with respect to the vertical line passing through $y$, and \eqref{eq:FKG}, we deduce that
\begin{equation}
  \phi_{S_{2n}}^0[x\longleftrightarrow x_1]\ge\phi_{\Lambda_{2n}}^0[x\lr y \text{ in $\Lambda_n$}]^2.
\end{equation}
Plugging the two previous displayed equations together, and then using \eqref{eq:FKG} for translates of the event on the left, we deduce that 
\begin{equation}
  \phi_{S_{2n}}^0[x\longleftrightarrow x_k]\ge\Big(\frac 1{cn^2}   \phi_{\Lambda_{2n}}^0[\mathcal H_{\Lambda_n}]\Big)^{2k},
\end{equation}
which, by letting $k$ tend to infinity, leads to the inequality 
\begin{equation}
  p_{2n}^2\ge \frac 1{cn^2}   \phi_{\Lambda_{2n}}^0[\mathcal H_{\Lambda_n}],
\end{equation}
In conclusion, $ \phi_{\Lambda_{2n}}^0[\mathcal H_{\Lambda_n}]$ also decays exponentially fast in $n$.

By Lemma~\ref{lem:1} applied to (say) $\lambda=3$, the sequence $(q_n)$ also decays exponentially fast. The reasoning above applied to the dual model implies that $1-\phi_{\Lambda_{2n}}^1[\mathcal V_{\Lambda_n}]$ decays exponentially fast in $n$. Therefore, \textbf{(DiscontCrit)} holds.

Now, assume (ii). We will show \eqref{eq:zz} which obviously implies {\bf (ContCrit)} and Corollary~\ref{cor:z} at once. By Lemma \ref{lem:1}, we also have $\inf q_n>0$. Fix $\rho>0$. Consider the rectangles
 \begin{align*}
 R&:=[0,\rho n]\times[0,n],\\
 \overline R&:=[-n,\rho n+n]\times[-n,2n],\\
 K&:=[-\tfrac23n,-\tfrac13n]\times[-n,2n],\\
K'&:=[ \rho n+\tfrac n3,\rho n+\tfrac{2n}3]\times[-n,2n].
 \end{align*} 
By \eqref{eq:DMP} and \eqref{eq:4}, we have that for every $\alpha\ge1$, $$\phi_{\overline R}^{\rm mix}[ \mathcal H_{R}]^{\alpha-1}\ge \phi_{S_{n}}^0[ \mathcal H_{[0,(\rho+2)\alpha n]\times[0,n]}],$$
where the mix boundary conditions are given by wired on left, wired on right, and free elsewhere.  Raising to the power $1/\alpha$ and letting $\alpha$ tend to infinity implies
\begin{equation}
  \phi_{\overline R}^{\rm mix}[ \mathcal H_{R}] \ge p_{n}^{\rho+2}. 
\end{equation}
The same reasoning as above with the dual gives
\begin{equation}
  \label{eq:45}
  \phi_{\overline R}^{\rm mix}[ \mathcal H_{K}^c\cap \mathcal H_{K'}^c|\mathcal H_{R}]\ge q_{n/3}^{18}.
\end{equation}
The two equations above imply 
\begin{equation}
  \label{eq:46}
  p_n^{\rho+2}q_{n/3}^{18}\le \phi_{\overline R}^{\rm mix}[ \mathcal H_{R}\cap\mathcal H_{K}^c\cap \mathcal H_{K'}^c ]\le \phi_{\overline R}^{\rm mix}[ \mathcal H_{R}| \mathcal H_{K}^c\cap \mathcal H_{K'}^c ]\le \phi_{\overline R}^0[\mathcal H_{R}],
\end{equation}
where the last inequality follows from \eqref{eq:DMP} and \eqref{eq:4}.

By duality and rotation invariance we also have $\phi_{\overline R}^1[\mathcal H_{R}]\le 1-q_n^{\rho+2}p_{n/3}^{18}.$
Finally,  \eqref{eq:CBC} concludes that for every boundary conditions~$\xi$,
\begin{equation}
  p_{n}^{\rho+2} q_{n/3}^{18}\le  \phi_{\overline R}^\xi [ \mathcal H_{R}] \le 1-q_n^{\rho+2}p_{n/3}^{18} .
\end{equation}
Therefore, (ii) implies \eqref{eq:zz}.

\section{Some simple applications of the properties in Theorem~\ref{thm:main}}

\subsection{Applications of \text{(SubCrit)} and \text{(SupCrit)}}
\label{sec:appl-sub-supcrit}

 In this section, set $0\longleftrightarrow\infty$ for the event that $0$ belongs to an infinite connected component.
\begin{proposition}\label{cor:2}
Assume {\bf (SubCrit)}. There exists $c>0$ such that for every $n\ge1$,
\begin{equation}\label{eq:kj}\phi_{\Lambda_n}^1[0\longleftrightarrow \partial\Lambda_n]\le \exp(-cn).\end{equation}
In particular, $\phi^1[0\longleftrightarrow\infty]=0$ and $\phi^0=\phi^1$.
\end{proposition}

\begin{proof}
By \eqref{eq:FKG}, we obtain that
$$\phi_{\Lambda_{2n}}^1[\calH_{\Lambda_n}]\ge \phi_{\Lambda_{2n}}^1[0\longleftrightarrow{\rm L}_{\Lambda_n}\text{ in }\Lambda_n]\cdot\phi_{\Lambda_{2n}}^1[0\longleftrightarrow{\rm R}_{\Lambda_n}\text{ in }\Lambda_n]\ge\tfrac1{16}\phi_{\Lambda_{2n}}[0\longleftrightarrow\partial\Lambda_n]^2.$$
In the second line, we used the union bound and the fact that, by invariance under rotations, the probability of being connected (in $\Lambda_n$) to the top, left, bottom or right sides is the same. The claim follows readily  since 
$$\phi_{\Lambda_{2n}}^1[0\longleftrightarrow\partial\Lambda_n]\ge \phi_{\Lambda_{2n}}^1[0\longleftrightarrow\partial\Lambda_{2n}]\ge \phi_{\Lambda_{2n+1}}^1[0\longleftrightarrow\partial\Lambda_{2n+1}].$$

By using \eqref{eq:4}, one can deduce that $\phi^1[0\longleftrightarrow\partial\Lambda_n]\le \exp(-cn)$. Letting $n$ go to infinity implies that $\phi^1[0\longleftrightarrow\infty]=0$. It is classical that the absence of infinite connected component implies $\phi^1=\phi^0$ (one can see it as a simple application of \eqref{eq:DMP} and \eqref{eq:4} that we leave as an interesting exercise; or see \cite{Gri06}). \end{proof}

\begin{proposition}\label{cor:3}
Assume {\bf (SupCrit)}. There exists $c>0$ such that for every $n\ge1$,
$$\phi_{\Lambda_n}^0[\Lambda_n\not\longleftrightarrow\infty]\le \exp(-cn).$$
In particular, $\phi^0[0\longleftrightarrow\infty]>0$ and $\phi^0=\phi^1$.
\end{proposition}
\begin{proof}
Recall that {\bf (SupCrit)} holds if and only if {\bf (SubCrit)} holds in the dual measure (in particular $\phi^1=\phi^0$). Since for $\Lambda_n$ not to be connected to a distance $n$, there must exist an open circuit in the dual configuration $\omega^*$ surrounding $\Lambda_n$. This circuit has length at least $n$ so that one can use \eqref{eq:kj} to conclude.
\end{proof}

\subsection{Applications of \textbf{(ContCrit)}}

\begin{corollary}[One-arm polynomial bound]\label{cor:a}
Assume {\bf (ContCrit)}. There exists $c>0$ such that for every $n\ge1$,
$$\frac{c}{n}\le \phi^1[0\longleftrightarrow\partial\Lambda_n]\le \frac1{n^c}.$$
In particular, $\phi^1[0\longleftrightarrow\infty]=0$ and $\phi^1=\phi^0$.
\end{corollary}

\begin{proof}
For
the lower bound, notice that for $\Lambda_{n}$ to be crossed horizontally, there must exist a vertex $x\in [-L,L]\times[-n,n]$  connected to a distance $n-L$ (where $L$ is the length of the longest edge in the graph). Thus, the union bound and the finite energy property (to relate the probability of neighboring vertices being connected to infinity) give
\begin{equation*}Cn\cdot\phi^1[0\longleftrightarrow\partial\Lambda_n]\ge \phi^1[\calH_{\Lambda_n}]\ge c_1\end{equation*}
where the constant $c_1>0$ (independent of $n$) is given by {\bf (ContCrit)}.

For the upper bound, we proceed as follows. If $\Lambda_n$ is connected to $\partial\Lambda_{4n}$, then one of the four rotated versions of $\calV_{[-3n,3n]\times[2n,3n]}$ must also occur (where the angles of the rotation are $\tfrac\pi2k$ with $0\le k\le 3$). Therefore,$$\phi^1_{\Lambda_{4n}\setminus\Lambda_n}\left[\partial\Lambda_n\longleftrightarrow \partial\Lambda_{4n}\right]\stackrel{\eqref{eq:FKG}}\le 1-(1-\phi^1_{\Lambda_{4n}\setminus\Lambda_n}[\calV_{[-3n,3n]\times[2n,3n]}])^4\le1- c_0,$$
where $c_0>0$ is given by Corollary~\ref{cor:z}. 
Successive applications of \eqref{eq:DMP} and \eqref{eq:4} imply the existence of $c>0$ such that 
\begin{equation}\label{eq:sup pol}\phi^1\left[0\longleftrightarrow\partial\Lambda_n\right]\le \prod_{4^k\le n}\phi^1_{\Lambda_{4^{k}}\setminus\Lambda_{4^{k-1}}}\left[\Lambda_{4^{k-1}}\longleftrightarrow \partial\Lambda_{4^k}\right]\le (1-c_0)^{\lfloor\log_4 n\rfloor}\le n^{-c},\end{equation}
which gives the right-hand side.

Letting $n$ go to infinity implies $\phi^1[0\longleftrightarrow\infty]=0$, which in turn implies $\phi^1=\phi^0$.
\end{proof}

\subsection{Proof of Corollary~\ref{cor:main}}

The proof will essentially consist in gathering known facts from this article and existing results.
\begin{itemize}
\item Corollary~\ref{cor:2} shows that \textbf{(SubCrit)} implies that $\phi^1[0\longleftrightarrow\infty]=0$. In particular, $p\le p_c(q)$. Conversely, it is now known in different ways that for $p<p_c(q)$, the probability of being connected to a  distance $n$ decays exponentially fast; see e.g.~\cite{BefDum12,DM,DumRaoTas17}. This fact implies that for $p<p_c(q)$, \textbf{(SubCrit)} holds.
Finally, Corollary~\ref{cor:1} shows that there exists $c>0$ such that \textbf{(SubCrit)} is equivalent to the assertion that there exists $n\ge1$ such that $\phi^1_{\Lambda_{2n}}[\mathcal H_n]<c$. Furthermore, by tracking the constants in the proofs, one easily sees that they can be taken to be continuous functions of $q$. This implies that the set of $(p,q)$ for which \textbf{(SubCrit)} occurs is an open subset of $\{(p,q):p\le p_c(q)\}$. Therefore, one must have that $q\mapsto p_c(q)$ is lower semi-continuous and that  $$\{(p,q):\text{\textbf{(SubCrit)}}\}=\{(p,q):p< p_c(q)\}.$$ 
\item By duality, $q\mapsto p_c(q)$ is upper semi-continuous and 
$$\{(p,q):\text{\textbf{(SupCrit)}}\}=\{(p,q):p> p_c(q)\}.$$
\item This shows that 
$$\{(p,q):\text{\textbf{(ContCrit)} or \textbf{(DiscontCrit)}}\}=\{(p,q):p= p_c(q)\}.$$
To conclude, observe that the proof of Theorem~\ref{thm:main} shows that the assertion {\em ``there exists $c=c(q)>0$ such that $(p_n)$ decays exponentially fast along the sequence $n_i=3^i$''} is equivalent to the assertion {\em ``there exists $n$ such that $p_n<c$''}. Keeping track of the dependency of $c$ on $q$, one can easily show that $c=c(q)$ can be taken to be a continuous function of $q$. This implies that the set of $(p,q)$ for which \textbf{(DiscontCrit)} occurs is an open subset of $\{(p,q):p= p_c(q)\}$.\end{itemize}

\begin{remark}
One can avoid using \cite{BefDum12,DM,DumRaoTas17} to conclude that {\bf (SubCrit)} and {\bf (SupCrit)} occur for $p<p_c(q)$ and $p>p_c(q)$ respectively. The fact that \textbf{(DiscontCrit)} implies $\phi^0\ne\phi^1$ forces $p$ to be equal to $p_c(q)$ (since for $p\ne p_c(q)$ the DLR-random-cluster measure is unique). It would therefore be sufficient to prove that \textbf{(ContCrit)} cannot hold for a non-trivial interval $[p_0,p_1]$ of values of $p$. This can be derived using sharp threshold theorems. Indeed, by Corollary~\ref{cor:a}, the probability to be connected to a distance $n$ decays at least polynomially. This classically implies that the influence (see \cite{Gri06} for a definition of the notion) for the event $\mathcal H_{\Lambda_n}$ is polynomially small. A use of a random-cluster version of the BKKKL result on influences, see \cite{GraGri}, enables one to show that the sequence of functions $p\mapsto \phi^1_{\Lambda_{2n},p}[\mathcal H_{\Lambda_n}]$ would undergo a sharp threshold, which would be contradictory with the fact that they remain bounded if they satisfy \textbf{(ContCrit)} for every $p$ in $[p_0,p_1]$. 
\end{remark}

\paragraph{Acknowledgments} 
We are grateful to  A.~Raoufi for valuable and enjoyable discussions during the project. This project was initiated during a visit at Princeton University in 2015. We thank the institution for the hospitality. The research of HDC is supported by an IDEX grant from Paris-Saclay and the ERC CriBLaM. The research of HDC and VT is supported by NCCR SwissMAP, funded by the Swiss NSF.

\newcommand{\etalchar}[1]{$^{#1}$}

\end{document}